\documentclass[review,onefignum,onetabnum]{siamart171218}
\usepackage{amsmath}
\usepackage{graphicx,color}
\usepackage{amssymb}
\usepackage{algorithm}
\usepackage{algorithmic}
\usepackage{float}

\usepackage{lipsum}
\usepackage{amsfonts}
\usepackage{graphicx}
\usepackage{epstopdf}
\usepackage{algorithmic}
\ifpdf
  \DeclareGraphicsExtensions{.eps,.pdf,.png,.jpg}
\else
  \DeclareGraphicsExtensions{.eps}
\fi

\newsiamremark{remark}{Remark}
\newsiamremark{hypothesis}{Hypothesis}
\crefname{hypothesis}{Hypothesis}{Hypotheses}
\newsiamthm{claim}{Claim}

\headers{Adaptive two-grid finite element algorithms}{Yukun Li and Yi Zhang}

\title{Analysis of adaptive two-grid finite element algorithms for linear and nonlinear problems}
\author{
Yukun Li\thanks{Department of Mathematics, University of Central Florida,
Orlando ({\tt yukun.li@ucf.edu}).}
\and Yi Zhang\thanks{Department of Mathematics and Statistics, The University of North Carolina at Greensboro, Greensboro ({\tt y\_zhang7@uncg.edu}).}
}

\usepackage{amsopn}

\def\O{\Omega}
\def\H{H_K}

\def\uc{u_{H_K^{k+1}}}
\def\uk{u_{H_K^{k}}}
\def\up{u_{H_K^{k-1}}}
\def\u0{u_{H_K^0}}

\def\W1inf{W^{1,\infty}}
\def\e2{\tilde{\eta}}
\def\T{\mathcal{T}}
\def\E{\mathcal{E}}
\def\R2K{\tilde{R}_K}
\def\J2E{\tilde{J}_E}
\def\[{[\![}
\def\]{]\!]}

\begin{document}

\maketitle


\begin{abstract}
This paper proposes some efficient and accurate adaptive two-grid (ATG) finite element algorithms for linear and nonlinear partial differential equations (PDEs). The main idea of these algorithms is to utilize the solutions on the $k$-th level adaptive meshes to find the solutions on the $(k+1)$-th level adaptive meshes which are constructed by performing adaptive element bisections on the $k$-th level adaptive meshes. These algorithms transform non-symmetric positive definite (non-SPD) PDEs (resp., nonlinear PDEs) into symmetric positive definite (SPD) PDEs (resp., linear PDEs). The proposed algorithms are both accurate and efficient due to the following advantages: they do not need to solve the non-symmetric or nonlinear systems; the degrees of freedom (d.o.f.) are very small; they are easily implemented; the interpolation errors are very small. Next, this paper constructs residue-type {\em a posteriori} error estimators, which are shown to be reliable and efficient. The key ingredient in proving the efficiency is to establish an upper bound of the oscillation terms, which may not be higher-order terms (h.o.t.) due to the low regularity of the numerical solution. Furthermore, the convergence of the algorithms is proved when bisection is used for the mesh refinements. Finally, numerical experiments are provided to verify the accuracy and efficiency of the ATG finite element algorithms, compared to regular adaptive finite element algorithms and two-grid finite element algorithms \cite{xu1996two}.
\end{abstract}

\begin{keywords}
adaptive algorithms, two-grid finite element algorithms, residual-type {\em a posteriori} error estimators, reliability, efficiency, convergence
\end{keywords}

\begin{AMS}
65N12, 
65N15, 
65N22, 
65N30, 
65M50,  
65N55  
\end{AMS}

\section{Introduction}\label{sec-1}
The two-grid discretization techniques for solving second order non-SPD linear PDEs and nonlinear PDEs were proposed by Xu \cite{xu1994novel, marion1995error, xu1996two}. In these algorithms, two spaces $\mathcal{V}_h$ and $\mathcal{V}_H$ are employed for the finite element discretization, with mesh size $h\ll H$. The idea of these algorithms is to first solve the original non-SPD linear PDE or nonlinear PDE on the coarser finite element space $\mathcal{V}_H$, and then find the solution $u_h$ of a linearized PDE on the finer finite element space $\mathcal{V}_h$ based on the coarser level solution $u_H$. The computational cost is significantly reduced since $\mathrm{dim}(\mathcal{V}_H)\ll\mathrm{dim}(\mathcal{V}_h)$, and the optimal accuracy can be maintained by choosing an appropriate coarser mesh size $H$. These two-grid techniques have been extended to solve mixed (Navier-)Stokes$\slash$Darcy model \cite{cai2009numerical, mu2007two}, time-harmonic Maxwell equations \cite{zhong2013two}, eigenvalue problems \cite{xu2001two}, etc.

The idea of adaptive methods is to homogenize the errors on all mesh elements, and to further improve the accuracy and efficiency of solving PDEs. The upper and lower bounds of the error estimator for general elliptic and parabolic PDEs were derived in \cite{verfurth1996review, verfurth2013posteriori} when the conforming finite element methods were used for the discretization. The convergence of the method and the bound of the convergence rate were studied in \cite{dorfler1996convergent}; however, there were some stringent restrictions on the initial mesh. 
The concept of data oscillation was introduced in \cite{morin2000data, morin2002convergence} to circumvent the requirements on the initial mesh. In \cite{binev2004adaptive}, a modification of the algorithm in \cite{morin2002convergence} was proposed and optimal estimates were proved by incorporating coarsening of the meshes. The convergence rate of the conforming finite element methods was also analyzed in \cite{cascon2008quasi}. For adaptive discontinuous Galerkin methods, the construction of different types of error estimators were introduced in \cite{karakashian2003posteriori}, and the convergence was investigated in \cite{karakashian2007convergence}. 

The residual-based {\em a posteriori} error estimates of two-grid finite element methods for nonlinear PDEs was proposed in \cite{bi2018posteriori}, where both coarser and finer meshes were quasi-uniform. The two-grid $hp$-version discontinuous Galerkin finite element methods were proposed for the second-order quasilinear elliptic PDEs of monotone type in \cite{congreve2013two}. We also refer the readers to references \cite{congreve2014, congreve2013two_2, congreve2014two, congreve2019two, congreve2013hp} for applying the two-grid $hp$-version discontinuous Galerkin finite element methods to other PDEs. In this paper, we propose some ATG finite element algorithms by using adaptive meshes as both the $k$-th level coarser meshes and the $(k+1)$-th level finer meshes. Based on the solutions on the $k$-th level coarser adaptive meshes, we only need to solve linear problems on the $(k+1)$-th level finer adaptive meshes. Specifically, there are three main objectives in this paper. First, we propose some ATG finite element algorithms, which incorporate the advantages of both two-grid finite element algorithms and adaptive finite element algorithms, i.e., no need to solve non-SPD/nonlinear PDEs and fewer degrees of freedom. Second, we design the {\emph a posteriori} error estimators and show that they are reliable, i.e., the error estimators provide upper bounds for the errors up to higher-order terms. Third, we prove the efficiency of the proposed estimators, and then establish the convergence of the ATG finite element algorithms. The main challenge in establishing these results arises from the low regularity of the numerical solutions so that the oscillation terms appeared in the {\emph a posteriori} error analysis may not be higher-order terms. To overcome this difficulty, we derive an upper bound for the oscillation terms (see Lemma \ref{lem20150623_11}). 

The organization of this paper is as follows. In Section \ref{sec-2}, some notations are introduced, and some preliminary results are stated. In Section \ref{sec-3}, the ATG finite element algorithm for non-SPD linear problems is considered. 
In Section \ref{sec-4}, several ATG finite element algorithms are proposed for the nonlinear problems. The reliability, the efficiency, and the convergence are proved. In Section \ref{sec-6}, some numerical tests are provided to verify the accuracy and efficiency of the proposed algorithms.

\section{Preliminaries}\label{sec-2} Through this paper, denote $\Omega$ as a convex polygonal domain in $\mathbb{R}^2$, and the standard Sobolev notations are used, i.e., for any set $B$,
\begin{alignat*}{2}
\|v\|_{L^p(B)}&=\bigg(\int_{B}|v|^pdx\bigg)^{1\slash p}\qquad &&1\le p<\infty,\\
|v|_{W^{m,p}(B)}&=\bigg(\sum_{|\alpha|=m}\|D^{\alpha}v\|_{L^p(B)}^p\bigg)^{1\slash p}\qquad &&1\le p<\infty,\notag\\
\|v\|_{W^{m,p}(B)}&=\bigg(\sum_{|\alpha|\le m}|D^{\alpha}v|_{L^p(B)}^p\bigg)^{1\slash p}\qquad  &&1\le p<\infty.
\end{alignat*}
When $m=0$, denote $W^{0,p}(B)$ by $L^p(B)$, and when $p=2$, denote $W^{m,2}(B)$ by $H^m(B)$. Also, define $H^1_0(\Omega)$ by $H^1_0(\Omega) := \{v\in H^1(\Omega): v|_{\partial\Omega}=0\}.$

We will consider two types of PDEs in this paper. The first type is non-SPD linear PDEs, and the second type is nonlinear PDEs. First, consider the non-SPD linear PDE with homogeneous Dirichlet boundary condition
\begin{alignat}{2}
-\mbox{\rm div\,}(\alpha(x)\nabla u) + \beta(x)\cdot\nabla u + \gamma(x)u &= 0\qquad\text{in}\quad&&\Omega,\label{eq20160622_3}\\
u&=0\qquad\text{on}\quad&&\partial\Omega.\label{eq20160622_4}
\end{alignat}
Here the coefficients $\alpha(x)\in\mathbb{R}^{2\times2}, \beta(x)\in\mathbb{R}^{2}$ and $\gamma(x)\in\mathbb{R}^{1}$ are assumed to be smooth on $\bar{\Omega}$, and $\alpha(x)$ satisfies, for some constant $C_1 > 0$, that
\begin{align}\label{eq20180326_1}
C_1|\xi|^2\le\xi^T\alpha(x)\xi \qquad\forall\xi\in\mathbb{R}^2. 
\end{align}
We also assume $\gamma(x) \geq 0$ in order to guarantee the solution $u$ of \eqref{eq20160622_3}--\eqref{eq20160622_4} is isolated.

Second, consider the following second order nonlinear problem 
\begin{alignat}{2}\label{eq20180219_5}
-\mathrm{div}(f(x,u,\nabla u))+g(x,u,\nabla u)&=0\qquad&&\text{in}\quad\Omega,\\
u&=0\qquad&&\text{on}\quad\partial\Omega,\label{eq20180219_6}
\end{alignat}
where $f(x,y,z): \overline{\Omega}\times\mathbb{R}^1\times\mathbb{R}^2\longrightarrow\mathbb{R}^2$ and $g(x,y,z): \overline{\Omega}\times\mathbb{R}^1\times\mathbb{R}^2\longrightarrow\mathbb{R}^1$ are smooth functions. Assume the solution of \eqref{eq20180219_5}--\eqref{eq20180219_6} satisfies $u\in H^1_0(\Omega)\cap W^{2,2+\kappa}(\Omega)$ for some $\kappa>0$. For any $w\in W^{1,\infty}(\Omega)$, we define the following notations:
\begin{alignat*}{2}
a(w) &= D_zf(x,w,\nabla w)\in\mathbb{R}^{2\times2},\qquad &&b(w)= D_yf(x,w,\nabla w)\in\mathbb{R}^2,\\
c(w) &= D_zg(x,w,\nabla w)\in\mathbb{R}^{2},\qquad &&d(w) = D_yg(x,w,\nabla w)\in\mathbb{R}^1.
\end{alignat*}
Let $\mathcal{L}$ be the operator defined by the left-hand side of \eqref{eq20180219_5}. 
The linearized operator $\mathcal{L}$ at $w$ is  
\begin{align}\label{eq20190820_1_add}
\mathcal{L}'(w) v = - \text{div} (a(w) \nabla v + b(w) v) + c(w) \cdot \nabla v + d(w) v. 
\end{align}

In the following, we introduce two parameters $\delta_1$ and $\delta_2$:
\begin{equation*}
\delta_1=\begin{cases}
0 & \text{if}\ D_z^2f(x,y,z)=0,\ D_z^2g(x,y,z)=0,\\
1 & \text{otherwise},
\end{cases}
\end{equation*}
and
\begin{equation*}
\delta_2=\begin{cases}
0 & \text{if}\ \delta_1=0,\ D_yD_zf(x,y,z)=0,\ D_yD_zg(x,y,z)=0,\\
1 & \text{otherwise}.
\end{cases}
\end{equation*}
In the case of $\delta_1 = 0$ and $\delta_2 = 1$, equation \eqref{eq20180219_5} becomes a mildly nonlinear PDE 
\begin{align}\label{eq20180219_7}
-\mbox{\rm div\,}(\widehat\alpha_1(x,u)\nabla u+\widehat\alpha_2(x,u)) + \widehat\beta(x,u)\cdot\nabla u +\widehat \gamma(x,u) = 0, 
\end{align}
which corresponds to $f(x,y,z) = \widehat\alpha_1(x,y) z + \widehat\alpha_2(x,y)$ and $g(x,y,z) = \widehat\beta(x,y) \cdot z + \widehat\gamma(x,y)$ in \eqref{eq20180219_5}.  For simplicity, we incorporate the term ${\rm div}(\widehat\alpha_2(x,u))$ into $\widehat\beta(x,u)\cdot\nabla u + \widehat\gamma(x,u)$. Therefore, \eqref{eq20180219_7} together with homogeneous Dirichlet boundary condition can be reduced to 
%
\begin{alignat}{2}
-\mbox{\rm div\,}(\widehat\alpha(x,u)\nabla u) + \widehat\beta(x,u)\cdot\nabla u + \widehat\gamma(x,u) &= 0\qquad\text{in}\quad&&\Omega,\label{eq20190702_2}\\
u&=0\qquad\text{on}\quad&&\partial\Omega, \label{eq20190702_3}
\end{alignat}
where we rewrite $\widehat\alpha_1(x,u)$ as $\widehat\alpha(x,u)$ and keep the notations of $\widehat\beta(x,u)$ and $\widehat\gamma(x,u)$. 

Similar to the non-SPD case, we assume 
\begin{align}\label{eq20190820_1}
C_2|\xi|^2\le\xi^T \widehat\alpha(x,u)\xi \qquad\forall\xi\in\mathbb{R}^2, 
\end{align}
for some constant $C_2 > 0$ and $- \text{div} (b(u)) + d(u) \geq 0$, which imply that $u$ is an isolated solution. 

We will focus on non-SPD linear PDEs and nonlinear PDEs in Section 3 and Section 4, respectively. In particular, we discuss mildly nonlinear PDEs with details and only describe algorithms for general nonlinear PDEs. 

\section{An adaptive two-grid finite element algorithm for non-SPD problems}\label{sec-3}
In this section, we present an adaptive two-grid finite element algorithm for non-SPD problems. The idea is to utilize the solutions on the $k$-th level adaptive meshes to transform the non-SPD problems into the SPD problems, and then to find the solutions of the SPD problems on the $(k+1)$-th level adaptive meshes. Denote $\mathcal{T}_k$ as the mesh in the $k$-th bisection and $H_K^k$ as the mesh size of $K$ in $\mathcal{T}_k$. 
Denote $\mathcal{E}_k$ as the set of mesh edges in the $k$-th bisection, $\mathcal{E}^i_k$ as the interior edges in $\mathcal{E}_k$, and $H_E^k$ as the size of the edge $E$ in $\mathcal{E}_k$.
 Then we define the $\mathcal{P}_r$-Lagrangian finite element space $\mathcal{V}_{H_K^k}$ on $\mathcal{T}_k$ below:
\begin{equation*}
\mathcal{V}_{H_K^k} = \bigl\{v_h \in H^1_0(\Omega): v_h|_{K} \in \mathcal{P}_r(K)\quad\forall K\in\mathcal{T}_k\bigr\},
\end{equation*}
where $\mathcal{P}_r$ denotes the space of all polynomials with degree less than or equal to $r$. Define the following notations
\begin{align*}
\widetilde{A}(u,v) &= A_S(u,v) + A_N(u,v), \\
A_S(u,v) &= (\alpha(x)\nabla u,\nabla v),\\
A_N(u,v) &= (\beta(x)\cdot\nabla u + \gamma(x)u,v).
\end{align*}
It is clear that the bilinear form $A_S(\cdot,\cdot)$ induces a norm on $H^1_0(\Omega)$, which is defined by 
\begin{align}\label{eq20160609_1}
|\|v\||_1^2:=A_S(v,v) \qquad \forall \, v \in H^1_0(\O).
\end{align}

The weak form of \eqref{eq20160622_3}-\eqref{eq20160622_4} is to seek $u\in H_0^1(\Omega)$ such that
\begin{alignat}{2}
\widetilde{A}(u,v) = 0 \qquad \forall \, v \in H^1_0(\Omega). 
\label{eq20160701_1}
\end{alignat}

Next, we present the adaptive two-grid finite element algorithm for \eqref{eq20160701_1}:
\begin{algorithm}[H]
\caption{The ATG finite element algorithm for non-SPD problems}
\label{algo1}
STEP 1: Find $u_{H_K^0}\in \mathcal{V}_{H_K^0}$ such that
\begin{align*}
\widetilde{A}(u_{H_K^0},v_{H_K^0}) = 0\qquad\forall v_{H_K^0}\in \mathcal{V}_{H_K^0};
\end{align*}
STEP 2: For $k\geq 0$, find $u_{H_K^{k+1}}\in \mathcal{V}_{H_K^{k+1}}$ such that
\begin{align*}
A_S(u_{H_K^{k+1}},v_{H_K^{k+1}}) + A_N(u_{H_K^k},v_{H_K^{k+1}}) = 0\qquad\forall v_{H_K^{k+1}}\in \mathcal{V}_{H_K^{k+1}}.
\end{align*}
\end{algorithm}

For each bisection $k$ in Step 2 of Algorithm \ref{algo1}, the error estimators (see Section \ref{sub_sec_3}) are computed to identify the elements that will be refined. Also, see Section 5 in the extended version of this paper in \cite{li2019analysis} for the implementation of the mesh refinement. In practice, the non-symmetric part $A_N(u_{H_K^k},v_{H_K^{k+1}})$ can be computed by interpolating $u_{H_K^k}$ from $\mathcal{T}_k$ to $\mathcal{T}_{k+1}$. 

\subsection{An {\em a posteriori} error estimate and convergence for the adaptive two-grid finite element algorithm \ref{algo1}}\label{sub_sec_3}
Denote $[\![ \cdot ]\!]$ as the jump of the function across the edges. Define the element residual on $K \in \T_k$ and the edge jump on $E \in \E^i_k$ by
\begin{align}
R_{K}^{k} &= -\mbox{div}(\alpha(x)\nabla u_{H_K^{k}}) + \beta(x)\cdot\nabla u_{H_K^{k-1}} + \gamma(x)u_{H_K^{k-1}}, \label{eq20160214_1}\\
J_{E}^{k} &= [\![\alpha(x)\nabla u_{H_K^{k}}]\!]_E.\label{eq20160214_2}
\end{align}
Note that we define $u_{H_K^{-1}} = u_{H_K^{0}}$ such that the above notations work for all $k \geq 0$. 
The local error estimators $\eta_{R,K}^k$ and $\eta_{J,E}^k$ are then defined by
\begin{align}
(\eta_{R,K}^k)^2&=(H_K^k)^2\|R_K^k\|_{L^2(K)}^2 \qquad \forall \, K \in \mathcal{T}_k,\label{eq20160214_3}\\
(\eta_{J,E}^k)^2&=H_E^k\|J_E^k\|_{L^2(E)}^2 \qquad \forall \, E \in \E^i_k.\label{eq20160214_4}
\end{align}
We define the global error estimators $\eta_{R}(u_{H_K^k},\mathcal{T}_k)$, $\eta_{J}(u_{H_K^k},\mathcal{T}_k)$ and $\eta(u_{H_K^k},\mathcal{T}_k)$ on the mesh $\mathcal{T}_k$ by
\begin{align}
\eta_{R}(u_{H_K^k},\mathcal{T}_k)&=\bigg(\sum_{K\in\mathcal{T}_k}(\eta_{R,K}^k)^2\bigg)^{1\slash2},\label{eq20160214_6}\\
\eta_{J}(u_{H_K^k},\mathcal{T}_k)&=\bigg(\sum_{E\in\mathcal{E}^i_k}(\eta_{J,E}^k)^2\bigg)^{1\slash2},\label{eq20160214_7}\\
\eta(u_{H_K^k},\mathcal{T}_k) &= \bigl((\eta_{R}(u_{H_K^k},\mathcal{T}_k))^2+(\eta_{J}(u_{H_K^k},\mathcal{T}_k))^2\bigr)^{1\slash2}.\label{eq20160214_8}
\end{align}

Let $\bar{R}_{K}^k$ and $\bar{J}_{E}^k$ be the $L^2$ projections to the piecewise $\mathcal{P}_{r-1}$ space respectively; then, define the oscillation terms:
\begin{align}
osc^R(u_{H_K^k},K) &:= H_K^k\|R_{K}^k-\bar{R}_{K}^k\|_{L^2(K)},\label{osc_1}\\
osc^J(u_{H_K^k},E) &:= (H_E^k)^{1\slash2}\|J_{E}^k-\bar{J}_{E}^k\|_{L^2(E)},\label{osc_2}\\
osc(u_{H_K^k},\mathcal{T}_k) &:= \bigg(\sum_{K\in\mathcal{T}_k}(osc^R(u_{H_K^k},K))^2+\sum_{E\in\mathcal{E}^i_k}(osc^J(u_{H_K^k},E))^2\bigg)^{1\slash2}.\label{osc_3}
\end{align}

A reliable upper bound of the error is given below. It shows that the error of the adaptive two-grid finite element algorithm \ref{algo1} can be bounded by the error estimator and a higher-order term. 
\begin{theorem}\label{thm20160701_1}
Let $u$ and $u_{H_K^{k+1}}$ be the solution of \eqref{eq20160701_1} and the adaptive two-grid finite element algorithm \ref{algo1}, respectively, and then
\begin{align*}
|\|u-u_{H_K^{k+1}}\||_1 \le C \eta(u_{H_K^{k+1}}, \T_{k+1}) + C\|u_{H_K^{k}}-u\|_{L^2(\Omega)}. 
\end{align*}
\end{theorem}
\begin{proof}
Let $v^I$ be the Scott-Zhang interpolation of $v$ on $\mathcal{V}_{H_K^{k+1}}$,  Then $\forall v\in H^1_0(\Omega)$, we have
\begin{align}\label{eq20160701_3}
&(\alpha(x)\nabla(u-u_{H_K^{k+1}}),\nabla v)\\
&= (-\beta(x)\cdot\nabla u - \gamma(x)u,v)-(\alpha(x)\nabla u_{H_K^{k+1}},\nabla v)\notag\\
&=(-\beta(x)\cdot\nabla u - \gamma(x)u,v)-(\alpha(x)\nabla u_{H_K^{k+1}},\nabla(v-v^I))\notag\\
&\qquad+(\beta(x)\cdot\nabla u_{H_K^{k}} + \gamma(x)u_{H_K^{k}},v^I)\notag\\
&=(\beta(x)\cdot\nabla(u_{H_K^{k}}- u) + \gamma(x)(u_{H_K^{k}}-u),v) \notag\\
&\quad-\sum_{E\in\mathcal{E}_{k+1}}([\![\alpha(x)\nabla(u_{H_K^{k+1}})]\!],v-v^I)+\sum_{K\in\mathcal{T}_{k+1}}(\mathrm{div}(\alpha(x)\nabla(u_{H_K^{k+1}})),v-v^I)\notag\\
&\quad-(\beta(x)\cdot\nabla u_{H_K^{k}} + \gamma(x)u_{H_K^{k}},v-v^I)\notag\\
&=-\sum_{K\in\mathcal{T}_{k+1}}(R_K^{k+1},v-v^I)-\sum_{K\in\mathcal{E}_{k+1}}(J_E^{k+1},v-v^I)\notag\\
&\quad+(\beta(x)\cdot\nabla(u_{H_K^{k}}- u) + \gamma(x)(u_{H_K^{k}}-u),v)\notag.
\end{align}

Using the Cauchy-Schwarz inequality, the trace inequality and properties of the Scott-Zhang interpolation operator, we have
\begin{align}\label{eq20160701_4}
&(\alpha(x)\nabla(u-u_{H_K^{k+1}}),\nabla v)\\
&\quad\le \bigg(\sum_{K\in\mathcal{T}_{k+1}}(H_K^{k+1})^2\|R_K^{k+1}\|_{L^2(K)}^2\bigg)^{\frac12}\bigg(\sum_{K\in\mathcal{T}_{k+1}}(H_K^{k+1})^{-2}\|v-v^I\|_{L^2(K)}^2\bigg)^{\frac12}\notag\\
&\quad+ C\bigg(\sum_{E\in\mathcal{E}_{k+1}}H_K^{k+1}\|J_E^{k+1}\|_{L^2(E)}^2\bigg)^{\frac12}\bigg(\sum_{E\in\mathcal{E}_{k+1}}(H_K^{k+1})^{-1}\|v-v^I\|_{L^2(E)}^2\bigg)^{\frac12}\notag\\
&\quad+(\beta(x)\cdot\nabla(u_{H_K^{k}}- u) + \gamma(x)(u_{H_K^{k}}-u),v)\notag\\
&\quad\le \eta_{R}(u_{H_K^{k+1}},\mathcal{T}_{k+1})\bigg(\sum_{K\in\mathcal{T}_{k+1}}(H_K^{k+1})^{-2}\|v-v^I\|_{L^2(K)}^2\bigg)^{\frac12}\notag\\
&\qquad+C\|u_{H_K^{k}}-u\|_{L^2(\Omega)}\|\nabla v\|_{L^2(\Omega)}\notag\\
&\qquad+C\eta_{J}(u_{H_K^{k+1}},\mathcal{T}_{k+1})\bigg(\sum_{E\in\mathcal{E}_{k+1}}(H_K^{k+1})^{-1}\|v-v^I\|_{L^2(E)}^2\bigg)^{\frac12}\notag\\
&\quad\le C\eta(u_{H_K^{k+1}}, \T_{k+1})\|\nabla v\|_{L^2(\Omega)}+C\|u_{H_K^{k}}-u\|_{L^2(\Omega)}\|\nabla v\|_{L^2(\Omega)}\notag.
\end{align}
The theorem is proved by choosing $v=u-u_{H_K^{k+1}}$.
\end{proof}

For the a posteriori error estimates of the classical finite element algorithms, the oscillation terms are usually higher-order terms, but for the proposed adaptive two-grid algorithm \ref{algo1}, the oscillation term may not be the higher-order term because of the low regularity of the solution. Next we will give an upper bound of the oscillation term, which plays crucial roles in proving the efficiency of the error estimator of the adaptive two-grid algorithm. As is seen below, it is bounded by the summation of the error and the higher-order term.
\begin{lemma}\label{lem20150623_11_add}
The oscillation term can be bounded by
\begin{align*}
osc(u_{H_K^{k+1}},\mathcal{T}_{k+1})&\leq e_1^{k+1}+e_2^{k+1}.
\end{align*}
where
\begin{align*}
e_1^{k+1}:=& \|\nabla (u - u_{H_K^{k+1}})\|_{L^2(\Omega)}+\bigg(\sum_{K\in\mathcal{T}_{k+1}}(H_K^{k+1})^2\|D^2u -D^2 u_{H_K^{k+1}}\|_{L^2(K)}^2\bigg)^\frac12\notag,\\
e_2^{k+1}:=&\|u - u_{H_K^k}\|_{L^2(\Omega)}+\bigg(\sum_{K\in\mathcal{T}_{k+1}}(H_K^{k+1})^2\|\nabla(u - u_{H_K^k})\|_{L^2(K)}^2\bigg)^\frac12\notag\\
& +\bigg(\sum_{K\in\mathcal{T}_{k+1}}(H_K^{k+1})^2\|\sigma^* - \bar{\sigma}^*\|_{L^2(K)}^2\bigg)^{\frac12}, \notag
\end{align*}
and $\sigma^*$ and $\bar{\sigma}^*$ are defined in the beginning of the proof.
\end{lemma}

\begin{proof}
For the residue estimator, define $\sigma^{k}$, $\hat{\sigma}^{k}$ and $\sigma^*$ by
\begin{align*}
\hat{\sigma}^k &:= -\alpha(x)^T:D^2u_{H_K^{k+1}}-\mbox{div}(\alpha(x)^T)\cdot\nabla u_{H_K^{k+1}} + \beta(x)\cdot\nabla u_{H_K^k} + \gamma(x)u_{H_K^k},\\
\sigma^* &:= -\alpha(x)^T:D^2u-\mbox{div}(\alpha(x)^T)\cdot\nabla u+ \beta(x)\cdot\nabla u + \gamma(x)u.
\end{align*}
Denote $\bar{\hat{\sigma}}^k$ and $\bar{\sigma}^*$ as the average of $\hat{\sigma}^k$ and $\sigma^*$ on $K$ respectively, and then
\begin{align}\label{eq20150623_12_add}
&\|\hat{\sigma}^k - \bar{\hat{\sigma}}^k\|_{L^2(K)}\\
\leq&\|\hat{\sigma}^k - \bar{\sigma}^*\|_{L^2(K)}\notag\\
\leq&\|\hat{\sigma}^k - \sigma^*\|_{L^2(K)}+\|\sigma^* - \bar{\sigma}^*\|_{L^2(K)}\notag\\
\leq&C\|D^2u - D^2u_{H_K^{k+1}}\|_{L^2(K)}+C\|\nabla (u - u_{H_K^{k+1}})\|_{L^2(K)}\notag\\
&\quad+C\|\nabla (u - u_{H_K^k})\|_{L^2(K)}+C\|u - u_{H_K^k}\|_{L^2(K)}+\|\sigma^* - \bar{\sigma}^*\|_{L^2(K)}.\notag
\end{align}
For the jump estimator, define $\hat{\delta}^{k}$ and $\delta$ by
\begin{align*}
\hat{\delta}^{k} &:= [\![\alpha(x)\nabla u_{H_K^{k+1}}]\!]_E,\\
\delta &:= [\![\alpha(x)\nabla u]\!]_E.
\end{align*}

Denote $\bar{\hat{\delta}}^k$ and $\bar{\delta}$ as the piecewise $L^2 $ projections of $\hat{\delta}^k$ and $\delta$ to $\mathcal{P}_{r-1}$, respectively.
Using the trace inequality with scaling and the triangle inequality, we have
\begin{align}\label{eq20160425_1_add}
&\|\hat{\delta}^k - \bar{\hat{\delta}}^k \|_{L^2(E)}\\
\leq & \| \hat{\delta}^k - \bar{{\delta}} \|_{L^2(E)} \notag\\
\leq &\| \hat{\delta}^k - {{\delta}} \|_{L^2(E)} + \| \delta  - \bar{{\delta}} \|_{L^2(E)} \notag \\
\leq &\| \[ \alpha(x) (\nabla \uc - \nabla u) \] \|_{L^2(E)} \notag \\
& + \| \[ (\alpha(x)-\alpha(x))\nabla u \] \|_{L^2(E)}+ \| \delta  - \bar{{\delta}} \|_{L^2(E)} \notag \\
\leq& C \sum_{K \in \omega_E} \Big( (\H^{k+1})^{-\frac12} \|\nabla(u - u_{H_K^{k+1}})\|_{L^2(K)}+ (\H^{k+1})^{\frac12}\|D^2u - D^2u_{H_K^{k+1}}\|_{L^2(K)} \notag \\
& + (\H^{k+1})^{-\frac12}   \|u-u_{H_K^k}\|_{L^2(K)}+C (\H^{k+1})^{\frac12}   \|\nabla(u-u_{H_K^k})\|_{L^2(K)} \Big) \notag\\
&+ \| \delta  - \bar{{\delta}} \|_{L^2(E)}, \notag 
\end{align}
where $\omega_E$ is the set of elements in $\T_{k+1}$ containing $E$ as an edge. 

Note that $\| \delta  - \bar{{\delta}} \|_{L^2(E)} = 0$ on $E \in \E^i_{k+1}$ since $u \in H^2(\Omega)$. Then using \eqref{eq20150623_12_add}, \eqref{eq20160425_1_add}, and the inverse inequality, we have
\begin{align}\label{eq20160425_3_add}
&\tilde{osc}(u_{H_K^{k+1}},\mathcal{T}_{k+1})\\
\leq &C\bigg(\sum_{K\in\mathcal{T}_{k+1}}(H_K^{k+1})^2\|D^2u -D^2 u_{H_K^{k+1}}\|_{L^2(K)}^2\bigg)^\frac12\notag\\
&+C\bigg(\sum_{K\in\mathcal{T}_{k+1}}(H_K^{k+1})^2\|\nabla(u - u_{H_K^{k+1}})\|_{L^2(K)}^2\bigg)^\frac12\notag\\
&+C\bigg(\sum_{K\in\mathcal{T}_{k+1}}(H_K^{k+1})^2\|\nabla(u - u_{H_K^k})\|_{L^2(K)}^2\bigg)^\frac12+C\|u - u_{H_K^k}\|_{L^2(\Omega)}\notag\\
&+C\bigg(\sum_{K\in\mathcal{T}_{k+1}}(H_K^{k+1})^2\|\sigma^* - \bar{\sigma}^*\|_{L^2(K)}^2\bigg)^{\frac12}+C\|\nabla(u - u_{H_K^{k+1}})\|_{L^2(\Omega)}\notag\\
\leq &C\|\nabla (u - u_{H_K^{k+1}})\|_{L^2(\Omega)}+C\bigg(\sum_{K\in\mathcal{T}_{k+1}}(H_K^{k+1})^2\|D^2u -D^2 u_{H_K^{k+1}}\|_{L^2(K)}^2\bigg)^\frac12\notag\\
&+C\|u - u_{H_K^k}\|_{L^2(\Omega)}+C\bigg(\sum_{K\in\mathcal{T}_{k+1}}(H_K^{k+1})^2\|\nabla(u - u_{H_K^k})\|_{L^2(K)}^2\bigg)^\frac12\notag\\
&+C\bigg(\sum_{K\in\mathcal{T}_{k+1}}(H_K^{k+1})^2\|\sigma^* - \bar{\sigma}^*\|_{L^2(K)}^2\bigg)^{\frac12} \notag \\ 
=& \tilde{e}_1^{k+1}+\tilde{e}_2^{k+1}\notag,
\end{align}
where $\tilde{e}_1^{k+1}, \tilde{e}_2^{k+1}$ denote the error term (the first two terms) and the higher-order terms (the last three terms) respectively.
\end{proof}

The following theorem gives a lower bound of the error of the adaptive two-grid finite element algorithm \ref{algo1}, i.e., the error is bounded below by the error estimator.  
\begin{theorem}\label{thm20160702_1}
Let $u$ and $u_{H_K^{k+1}}$ be the solution of \eqref{eq20160701_1} and the adaptive two-grid finite element algorithm \ref{algo1}, respectively, and then 
\begin{equation*}
\eta(u_{H_K^{k+1}}, \T_{k+1}) \leq C(e_1^{k+1}+e_2^{k+1}).
\end{equation*}
\end{theorem}
\begin{proof}
We divide the proof into three steps:\\
Step 1: Using the properties of the element bubble functions $\varphi_K$, we have
\begin{align}\label{eq20160702_5}
&\quad\frac{9}{20}\|\bar{R}_K^{k+1}\|_{L^2(K)}^2\\
&=(\bar{R}_K^{k+1},\varphi_K\bar{R}_K^{k+1})\notag\\
&=(R_K^{k+1},\varphi_K\bar{R}_K^{k+1})-(R_K^{k+1}-\bar{R}_K^{k+1},\varphi_K\bar{R}_K^{k+1})\notag\\
&=(\beta(x)\cdot\nabla (u_{H_K^k}-u) + \gamma(x)(u_{H_K^k}-u),\varphi_K\bar{R}_K^{k+1})\notag\\
&\quad+(\alpha(x)\nabla (u_{H_K^{k+1}}-u),\nabla(\varphi_K\bar{R}_K^{k+1}))-(R_K^{k+1}-\bar{R}_K^{k+1},\varphi_K\bar{R}_K^{k+1})\notag\\
&\le C(H_K^{k+1})^{-1}\|u_{H_K^k}-u\|_{L^2(K)}\|\bar{R}_K^{k+1}\|_{L^2(K)}+C(H_K^{k+1})^{-1}\notag\\
&\quad\cdot|\|u_{H_K^{k+1}}-u\||_1\|\bar{R}_K^{k+1}\|_{L^2(K)}+C\|R_K^{k+1}-\bar{R}_K^{k+1}\|_{L^2(K)}\|\bar{R}_K^{k+1}\|_{L^2(K)}.\notag
\end{align}
By \eqref{eq20160702_5} and the triangle inequality, we get
\begin{align}\label{eq20160702_6}
H_K^{k+1}\|R_K^{k+1}\|_{L^2(K)}\le &C\|u_{H_K^k}-u\|_{L^2(K)}+C|\|u_{H_K^{k+1}}-u\||_1\\
&+CH_K^{k+1}\|R_K^{k+1}-\bar{R}_K^{k+1}\|_{L^2(K)}.\notag
\end{align}
Step 2: Using the properties of the edge bubble functions $\psi_K$, we have
\begin{align}\label{eq20160702_7}
&\quad\frac{2}{3}\|\bar{J}_E^{k+1}\|_{L^2(E)}^2\\
&=(\bar{J}_E^{k+1},\psi_K\bar{J}_E^{k+1})_E\notag\\
&=(J_E^{k+1},\psi_K\bar{J}_E^{k+1})_E-(J_E^{k+1}-\bar{J}_E^{k+1},\psi_K\bar{J}_E^{k+1})_E\notag\\
&=(\alpha(x)\nabla u_{H_K^{k+1}},\nabla(\psi_K\bar{J}_E^{k+1}))+(\mathrm{div}(\alpha(x)\nabla u_{H_K^{k+1}}),\psi_K\bar{J}_E^{k+1}))\notag\\
&\quad-(J_E^{k+1}-\bar{J}_E^{k+1},\psi_K\bar{J}_E^{k+1})_E\notag\\
&=(\alpha(x)\nabla(u_{H_K^{k+1}}-u),\nabla(\psi_K\bar{J}_E^{k+1}))+(\beta(x)\cdot\nabla(u_{H_K^k}-u),\psi_K\bar{J}_E^{k+1})\notag\\
&\quad+(\gamma(x)(u_{H_K^k}-u),\psi_K\bar{J}_E^{k+1})-(R_K^{k+1},\psi_K\bar{J}_E^{k+1})\notag\\
&\quad-(J_E^{k+1}-\bar{J}_E^{k+1},\psi_K\bar{J}_E^{k+1})_E\notag\\
&\le C(H_K^{k+1})^{-\frac12}\|u_{H_K^{k+1}}-u)\|_{H^1(w_E)}\|\bar{J}_E^{k+1}\|_{L^2(E)}\notag\\
&\quad+C(H_K^{k+1})^{-\frac12}\|u_{H_K^k}-u\|_{L^2(w_E)}\|\bar{J}_E^{k+1}\|_{L^2(E)}\notag\\
&\quad+(H_K^{k+1})^{\frac12}\|R_K^{k+1}\|_{L^2(w_E)}\|\bar{J}_E^{k+1}\|_{L^2(E)}+C\|J_K^k-\bar{J}_K^k\|_{L^2(E)}\|\bar{J}_E^{k+1}\|_{L^2(E)}\notag,
\end{align}
where $w_E$ denotes all the triangles contain $E$ as an edge.

By \eqref{eq20160702_6} and triangle inequality, we have
\begin{align}\label{eq20160703_1}
&(H_K^{k+1})^{\frac12}\|J_K^{k+1}\|_{L^2(E)}\le C\|u_{H_K^k}-u\|_{L^2(w_E)}+C\|u_{H_K^{k+1}}-u\|_{H^1(w_E)}\\
&\qquad+C(H_K^{k+1})^{\frac12}\|J_K^{k+1}-\bar{J}_K^{k+1}\|_{L^2(E)}+CH_K^{k+1}\|R_K^{k+1}-\bar{R}_K^{k+1}\|_{L^2(K)}.\notag
\end{align}
Step 3: Summing \eqref{eq20160702_6} over all $K$, and then
\begin{align}\label{eq20160703_2}
\eta_{R}(u_{H_K^k},\mathcal{T}_k)^2\le &C\|u_{H_K^k}-u\|_{L^2(\Omega)}^2+C|\|u_{H_K^{k+1}}-u\||_1^2\\
&+C\sum_{K\in\mathcal{T}_{k+1}}(osc^R(u_{H_K^{k+1}},\mathcal{T}_{k+1}))^2.\notag
\end{align}
 
Summing \eqref{eq20160703_1} over all $E$, and then
\begin{align}\label{eq20160703_3}
&\eta_{J}(u_{H_K^k},\mathcal{T}_k)^2\le C\|u_{H_K^k}-u\|_{L^2(\Omega)}^2+C|\|u_{H_K^{k+1}}-u\||_1^2\\
&\qquad+\sum_{E\in\mathcal{E}_{k+1}}(osc^J(u_{H_K^{k+1}},\mathcal{T}_{k+1}))^2+C\sum_{K\in\mathcal{T}_{k+1}}(osc^R(u_{H_K^{k+1}},\mathcal{T}_{k+1}))^2.\notag
\end{align}

Combining \eqref{eq20160703_2}, \eqref{eq20160703_3} and Lemma \ref{lem20150623_11_add}, the theorem is proved.
\end{proof}


The next objective is to show the convergence of the error of the ATG finite element algorithm \ref{algo1}, i.e., the errors decrease to zero up to higher-order terms. We need the following bulk criterion (proposed by D\"{o}rfler \cite{dorfler1996convergent}), i.e., we mark the element set $\mathcal{M}_k \subset \T_k$ such that 
\begin{align}\label{20150623_7_add}
\eta^2(\uk, \mathcal{M}_k) \geq \theta \eta^2(\uk, \T_k), 
\end{align}
where $\theta\in(0,1)$ is a constant which implies the number of marked elements. 

The following two lemmas are useful in the proof of the error reduction property. The first lemma is extracted from \cite{hu2013convergence}, and we include its proof here for the convenience of the readers. 
\begin{lemma}\label{lem:conv1}
Define $\rho=1-\frac{1}{\sqrt{2}}$, we have 
\begin{align*}
\e2^2(u_{H_K^k},\mathcal{T}_{k+1})\leq \e2^2(u_{H_K^k},\mathcal{T}_{k})-\rho\e2^2(u_{H_K^k},\mathcal{T}_{k}\backslash \mathcal{T}_{k+1}).
\end{align*}
\end{lemma}
\begin{proof}
By the definition of the estimator, we have
\begin{align}\label{20150623_1}
\e2^2(u_{H_K^k},\mathcal{T}_{k+1})= \e2^2(u_{H_K^k},\mathcal{T}_{k+1}\cap\mathcal{T}_{k})+\e2^2(u_{H_K^k},\mathcal{T}_{k+1}\backslash \mathcal{T}_{k}),
\end{align}
For any $K \in \T_k \setminus \T_{k+1}$, assume K is subdivided into $K=K^1\cup K^2$ with $K^1, K^2\in \mathcal{T}_{k+1}$ and $|K^1|=|K^2|=\frac12|K|$.
It is easy to show that
\begin{align}\label{20150623_2}
\sum_{i=1}^2\e2^2(u_{H_K^k},K^i)\leq\frac{1}{\sqrt{2}}\e2^2(u_{H_K^k},K),
\end{align}
Therefore, we have
\begin{align}\label{20150623_3}
\sum_{K^i \in\mathcal{T}_{k+1}\backslash \mathcal{T}_{k}}\sum_{i=1}^2\e2^2(u_{H_K^k},K^i)\leq\frac{1}{\sqrt{2}}\e2^2(u_{H_K^k},\mathcal{T}_{k}\backslash \mathcal{T}_{k+1}).
\end{align}
Pluging \eqref{20150623_3} into \eqref{20150623_1}, we get
\begin{align}\label{eq20180224_1}
\e2^2(u_{H_K^k},\mathcal{T}_{k+1})&\leq \e2^2(u_{H_K^k},\mathcal{T}_{k+1}\cap\mathcal{T}_{k})+\frac{1}{\sqrt{2}}\e2^2(u_{H_K^k},\mathcal{T}_{k}\backslash \mathcal{T}_{k+1}),\\
&\leq \e2^2(u_{H_K^k},\mathcal{T}_{k})-\rho\e2^2(u_{H_K^k},\mathcal{T}_{k}\backslash \mathcal{T}_{k+1}).\notag
\end{align}
The lemma is proved.
\end{proof}

\begin{lemma}\label{lem20150623_2}
Assume the bulk criterion \eqref{20150623_7_add} holds, and then for any $\epsilon>0$, there exists $\beta_1(\epsilon) > 0$ such that 
\begin{align*}
\eta^2(u_{H^{k+1}_K},\mathcal{T}_{k+1})\leq&(1+\epsilon)(1-\rho\theta)\eta^2(u_{H_K^k},\mathcal{T}_{k})\\
&+\frac{1}{\beta_1(\epsilon)}\|\nabla(u_{H^{k+1}_K}-u_{H_K^k})\|_{L^2(\Omega)}^2+C\|u_{H^{k}_K}-u_{H_K^{k-1}}\|_{L^2(\Omega)}^2.\notag
\end{align*}
\end{lemma}

\begin{proof}
Using the definition of the estimator, the inverse inequality, and the trace inequality with scaling, we have
\begin{align}\label{20150623_5}
&|\eta(u_{H^{k+1}_K},\mathcal{T}_{k+1})-\eta(u_{H_K^k},\mathcal{T}_{k+1})|\\
=& \bigg| \bigg(\sum_{K\in\mathcal{T}_{k+1}}(H^{k+1}_K)^2\|-\mbox{div}(\alpha(x)\nabla u_{H^{k+1}_K}) + \beta(x)\cdot\nabla u_{H_K^{k}} + \gamma(x)u_{H_K^{k}}\|_{L^2(K)}^2 \notag\\
&\qquad\qquad+ \sum_{E\in\mathcal{E}_{k+1}}H^{k+1}_K\bigl\|[\![\alpha(x)\nabla u_{H^{k+1}_K}\cdot\mathbf{n}]\!]\bigr\|_{L^2(E)}^2 \bigg)^{1\slash2} \notag\\
&-\bigg(\sum_{K\in\mathcal{T}_{k+1}}(H^{k+1}_K)^2\|-\mbox{div}(\alpha(x)\nabla u_{H_K^{k}}) + \beta(x)\cdot\nabla u_{H_K^{k-1}} + \gamma(x)u_{H_K^{k-1}}\|_{L^2(K)}^2 \notag\\
&\qquad\qquad+ \sum_{E\in\mathcal{E}_{k+1}}H^{k+1}_K\bigl\|[\![\alpha(x)\nabla u_{H_K^k}\cdot\mathbf{n}]\!]\bigr\|_{L^2(E)}^2 \bigg)^{1\slash2} \notag\\
\leq& \sum_{K\in\mathcal{T}_{k+1}}(H^{k+1}_K)^2\|-\mbox{div}(\alpha(x)\nabla u_{H^{k+1}_K}) + \mbox{div}(\alpha(x) \nabla \uk) \|^2_{L^2(K)} \notag \\
&\qquad\qquad + \sum_{K\in\mathcal{T}_{k+1}}(H^{k+1}_K)^2 \|\beta(x) \cdot \nabla (\uk - \up) \|^2_{L^2(K)} \notag \\
&\qquad\qquad + \sum_{K\in\mathcal{T}_{k+1}}(H^{k+1}_K)^2 \| \gamma(x) (\uk - \up)\|^2_{L^2(K)} \notag \\
&\qquad\qquad + \sum_{E\in\mathcal{E}_{k+1}}H^{k+1}_K\bigl\|[\![\alpha(x)\nabla (u_{H^{k+1}_K - \uk)}\cdot\mathbf{n}]\!]\bigr\|_{L^2(E)}^2 \notag \\
\leq& C\|\nabla(u_{H^{k+1}_K}-u_{H_K^k})\|_{L^2(\Omega)}+C\|u_{H^{k}_K}-u_{H_K^{k-1}}\|_{L^2(\Omega)}. \notag
\end{align}

Combining \eqref{20150623_5}, Young's inequality, and Lemma \ref{lem:conv1}, we get
\begin{align}\label{20150623_6}
\eta^2(u_{H^{k+1}_K},\mathcal{T}_{k+1})\leq& (1+\epsilon)\eta^2(u_{H_K^k},\mathcal{T}_{k})-(1+\epsilon)\rho\eta^2(u_{H_K^k},\mathcal{T}_k\backslash\mathcal{T}_{k+1})\\
&+\frac{1}{\beta_1(\epsilon)}\|\nabla(u_{H^{k+1}_K}-u_{H_K^k})\|_{L^2(\Omega)}^2+C\|u_{H^{k}_K}-u_{H_K^{k-1}}\|_{L^2(\Omega)}^2\notag, 
\end{align}
where $\beta_1(\epsilon) > 0$ and $\epsilon > 0$.

Now applying the bulk criterion in \eqref{20150623_7_add}:
\begin{align}
\eta^2(u_{H_K^k},\mathcal{T}_{k}\backslash\mathcal{T}_{k+1})\geq\theta \eta^2(u_{H_K^k},\mathcal{T}_{k}),
\end{align}
where $0<\theta<1$,
we then have 
\begin{align}\label{20150623_8}
\eta^2(u_{H^{k+1}_K},\mathcal{T}_{k+1})\leq&(1+\epsilon)(1-\rho\theta)\eta^2(u_{H_K^k},\mathcal{T}_{k})\\
&+\frac{1}{\beta_1(\epsilon)}\|\nabla(u_{H^{k+1}_K}-u_{H_K^k})\|_{L^2(\Omega)}^2+C\|u_{H^{k}_K}-u_{H_K^{k-1}}\|_{L^2(\Omega)}^2\notag, 
\end{align}
which complete the proof. 
\end{proof}

Based on Lemmas \ref{lem:conv1}--\ref{lem20150623_2}, we prove the quasi error decreases with respect to the number of mesh bisections up to some $L^2$-norms of the errors, which are higher-order terms on the uniform meshes.
\begin{theorem} (Error reduction)\label{thm20160214_4}
The following error reduction property holds
\begin{align*}
|\|u-u_{H^{k+1}_K}\||_1^2 + C\eta^2(u_{H^{k+1}_K},\mathcal{T}_{k+1})\leq&\zeta(\epsilon)\bigl(|\|u-u_{H_K^k}\||_1^2 + C\eta^2(u_{H_K^k},\mathcal{T}_{k})\bigr)\\
&+C\|u-u_{H_K^k}\|_{L^2(\Omega)}^2+C\|u-u_{H_K^{k-1}}\|_{L^2(\Omega)}^2, 
\end{align*}
for some $0 < \zeta(\epsilon) < 1$. 
\end{theorem}
\begin{proof}
By the Galerkin orthogonality, we have
\begin{align}\label{eq20160426_1}
|\|u_{H^{k+1}_K}-u_{H_K^k}\||_1^2 = |\|u-u_{H_K^k}\||_1^2 - |\|u-u_{H^{k+1}_K}\||_1^2\\
\qquad-(\alpha(x)\nabla(u-u_{H^{k+1}_K}),\nabla(u_{H^{k+1}_K}-u_{H_K^k})).\notag
\end{align}
Using Young's inequality and Poincar$\acute{e}$'s inequality, we get
\begin{align}\label{eq20160609_3}
&\qquad-A_S(u-u_{H^{k+1}_K},u_{H^{k+1}_K}-u_{H_K^k})\\
&=A_N(u,u_{H^{k+1}_K}-u_{H_K^k})-(\beta(x)\cdot\nabla u_{H_K^k} + \gamma(x)u_{H_K^k},u_{H^{k+1}_K}-u_{H_K^k})\notag\\
&\leq \delta_1 |\| u - \uc \||_1^2 + \delta_1 |\| u - \uk \||_1^2 + C(\delta_1) \| u - \uk \|_{L^2(\O)}^2, \notag 
\end{align}
for some $0 < \delta_1 < 1$ that will be determined later. 

Combining \eqref{eq20160426_1} and \eqref{eq20160609_3}, we have
\begin{align}\label{eq20160426_4}
(1 - \delta_1) |\| u - \uc \||_1^2 &\leq (1 + \delta_1) |\| u - \uk \||_1^2 - \| \uc - \uk \|^2_{H^1(\O)} \\
&\quad + C(\delta_1) \| u - \uk \|^2_{L^2(\O)}. \notag 
\end{align}

Using \eqref{eq20160426_4}, Lemma \ref{lem20150623_2}, 
denoting $\beta_2(\epsilon) = (1+\epsilon)(1-\rho\theta) < 1$,  
and choosing $\gamma_1, \gamma_2 > 0$ such that $\frac{\gamma_2}{\beta_1(\epsilon)} - \gamma_1 = 0$, we have 
\begin{align}\label{eq20160609_5}
&\gamma_1 (1-\delta_1)|\|u-u_{H^{k+1}_K}\||_1^2 + \gamma_2 \eta^2(u_{H^{k+1}_K},\mathcal{T}_{k+1})\\
&\qquad\leq \gamma_1(1+\delta_1) |\|u-u_{H_K^k}\||_1^2 + \gamma_2 \beta_2(\epsilon) \eta^2(u_{H_K^k},\mathcal{T}_{k})\notag\\
&\qquad\qquad+ \big(\frac{\gamma_2}{\beta_1(\epsilon)} - \gamma_1\big) \|u_{H^{k}_K}-u_{H_K^{k-1}}\|_{L^2(\Omega)}^2 + C \gamma_2 \| \uk - \up \|^2_{L^2(\O)} \notag\\
&\qquad\qquad + C(\delta_1) \gamma_1 \| u - \uk \|^2_{L^2(\O)} \notag \\ 
&\qquad = \zeta(\epsilon) \left[ \gamma_1(1-\delta_1) |\|u-u_{H_K^k}\||_1^2 + \gamma_2 \eta^2(u_{H_K^k},\mathcal{T}_{k}) \right] \notag \\
&\qquad\qquad + \gamma_1 \left[ (1+\delta_1) - \zeta(\epsilon) (1-\delta_1) \right] \|u - u_{H^{k}_K}\|_{L^2(\Omega)}^2 \notag \\
&\qquad\qquad + \gamma_2 (\beta_2(\epsilon) - \zeta(\epsilon)) \eta^2(u_{H_K^k},\mathcal{T}_{k}) \notag \\ 
&\qquad\qquad + C \gamma_2 \| \uk - \up \|^2_{L^2(\O)} + C(\delta_1) \gamma_1 \| u - \uk \|^2_{L^2(\O)}, \notag
\end{align}
where $0<\zeta(\epsilon)<1$ will be determined later. Applying the reliability of the estimator in Lemma \ref{thm20160701_1} below
\begin{align}\label{add1}
|\| u - \uk \||_1^2 \leq C_{Rel} \eta^2(\uk, \T_k) + C_{Rel} \| u - \up \|^2_{L^2(\O)}, 
\end{align}
and inserting this into the right-hand side of \eqref{eq20160609_5}, we have 
\begin{align}\label{add2}
\gamma_1 (1-\delta_1) &|\|u-u_{H^{k+1}_K}\||_1^2 + \gamma_2 \eta^2(u_{H^{k+1}_K},\mathcal{T}_{k+1}) \notag \\
&\leq  \zeta(\epsilon) \left[ \gamma_1(1-\delta_1) |\|u-u_{H_K^k}\||_1^2 + \gamma_2 \eta^2(u_{H_K^k},\mathcal{T}_{k}) \right] \\
&\quad + \big[ \gamma_1 [(1+\delta_1) - \zeta(\epsilon)(1-\delta_1)] C_{Rel} + \gamma_2 (\beta_2(\epsilon) - \zeta(\epsilon)) \big] \eta^2(\uk, \T_k) \notag \\ 
&\quad + C \| \uk - \up \|^2_{L^2(\O)} + C \| u - \uk \|^2_{L^2(\O)} + C \| u - \up \|^2_{L^2(\O)}. \notag 
\end{align}
We can choose $\delta_1>0$ and $\zeta(\epsilon)$ according to 
\begin{align}\label{add3}
\zeta(\epsilon) &= \frac{\gamma_1(1+\delta_1)C_{Rel} + \gamma_2 \beta_2(\epsilon)}{C_{Rel} (1-\delta_1) \gamma_1 + \gamma_2}, \\
\delta_1 &= \min \left\{ 1, \frac{\gamma_2(1-\beta_2(\epsilon))}{C_{Rel} \gamma_1} \right\}, 
\end{align}
such that $0 < \zeta(\epsilon) < 1$ and the second term on the right-hand side of \eqref{add2} vanishes. 

Finally, the estimate follows from a simple re-scaling of the coefficient in \eqref{add3}. 
\end{proof}

The following convergence result shows that the error of the ATG finite element algorithm \ref{algo1} decreases to zero up to higher-order terms, whose proof involves the reliability result in Theorem \ref{thm20160701_1} and the bulk criterion \eqref{20150623_7_add}.

\begin{theorem} (Convergence)\label{thm20160214_5}
Let $u$ and $u_{H_K^{k+1}}$ be the solution of \eqref{eq20160701_1} and the adaptive two-grid finite element algorithm \ref{algo1}, respectively. Assume the bulk criterion \eqref{20150623_7_add} holds, and then there exist $\rho > 0$ and $0 < \zeta < 1$ such that 
\begin{align*}
&|\|u-u_{H^{k+1}_K}\||_1^2 + \rho \eta^2(u_{H^{k+1}_K},\mathcal{T}_{k+1})\\
&\qquad\leq \zeta^{k+1}\bigl(|\|u-u_{H_K^0}\||_1^2+\rho\eta^2(u_{H_K^0},\mathcal{T}_0)\bigr)\\
&\qquad\quad+C\sum_{i=1}^{k}\|u-u_{H_K^{k-i}}\|_{L^2(\Omega)}^2\zeta^i+C\sum_{i=1}^{k}\|u-u_{H_K^{k-1-i}}\|_{L^2(\Omega)}^2\zeta^i\notag.
\end{align*}
\end{theorem}
\begin{proof}
Define $\|u-u_{H_K^{-1}}\|_{L^2(\Omega)}^2=\|u-u_{H_K^{0}}\|_{L^2(\Omega)}^2$. Using Theorem \ref{thm20160214_4}, we have
\begin{align}\label{eq20160426_3}
&|\|u-u_{H^{k+1}_K}\||_1^2 + C\eta^2(u_{H^{k+1}_K},\mathcal{T}_{k+1})\\
\leq&(\zeta(\epsilon))^{k+1}\bigl(|\|u-u_{H_K^0}\||_1^2+C\eta^2(u_{H_K^0},\mathcal{T}_0)\bigr)\notag\\
&+C\sum_{i=1}^{k}\|u-u_{H_K^{k-i}}\|_{L^2(\Omega)}^2(\zeta(\epsilon))^i+C\sum_{i=1}^{k}\|u-u_{H_K^{k-1-i}}\|_{L^2(\Omega)}^2(\zeta(\epsilon))^i\notag.
\end{align}
\end{proof}


\section{Adaptive two-grid finite element algorithms for nonlinear PDEs}\label{sec-4}
In this section, we present some algorithms for solving nonlinear PDEs. The idea is to transform nonlinear PDEs into linear ones using the coarser level solutions. 
For simplicity, we give the proof of the reliability, efficiency, and convergence of Algorithm \ref{algo3}. 

Define $A(u,v)$ by 
\begin{align}\label{eq20190820_3}
A(u,v) = (f(x,u,\nabla u), \nabla v) + (g(x,u,\nabla u), v). 
\end{align}
It is clear that the weak form of the nonlinear PDE \eqref{eq20190702_2}--\eqref{eq20190702_3} is given by 
\begin{align}\label{eq20190820_4}
A(u,v) = 0 \qquad \forall \, v \in H^1_0(\Omega). 
\end{align}
By \eqref{eq20190820_1}, the following energy norm is equivalent to $\| \cdot \|_{H^1(\O)}$ \begin{align}\label{norm:conv2}
|\| v \||_2^2 := (\widehat\alpha(x,u) \nabla v, \nabla v) \qquad \forall \, v \in H^1_0(\O). 
\end{align}
The following notation will be used in the subsequent subsections 
\begin{align*}
A_2(w,v,\xi) := (a(w)\nabla v + b(w)v,\nabla\xi)+(c(w)\cdot\nabla v+d(w)v,\xi).
\end{align*}

\subsection{Mildly nonlinear PDEs}
The solution $u$ of mildly nonlinear PDE \eqref{eq20190702_2}--\eqref{eq20190702_3} satisfies 
\begin{align*}
\widehat{A}(u,v) := (\widehat\alpha(x,u)\nabla u,\nabla v) + (\widehat\beta(x,u)\cdot\nabla u + \widehat\gamma(x,u),v) = 0 \qquad \forall \, v \in H^1_0(\Omega). 
\end{align*}
For $w,v,\xi\in W^{1,\infty}(\Omega)\cap H^1_0(\Omega)$, we define
\begin{align}\label{eqq20190702_1}
A_1(w,v,\xi) := (\widehat\alpha(x,w)\nabla v,\nabla\xi) + (\widehat\beta(x,w)\cdot\nabla v + \widehat\gamma(x,w),\xi).
\end{align}

%
%

Similar to Algorithm \ref{algo1}, the solutions on the coarser level meshes can be used to transform the nonlinear PDEs into linear ones. Algorithm \ref{algo3} is proposed by directly substituting the coefficients of the nonlinear PDE with coarser grid solutions. The error estimator will be presented in Section~\ref{subsec4.1.1}. 
Note that the initial mesh is usually chosen to be the uniform mesh, i.e., $H_K^0=H^0$ for all $K \in \T_0$. 

\begin{algorithm}[H]
\caption{The ATG finite element algorithm for mildly nonlinear problems}
\label{algo3}
STEP 1: Find $u_{H_K^0}\in \mathcal{V}_{H_K^0}$ such that
\begin{align*}
\widehat{A}(u_{H_K^0},v_{H_K^0}) = 0\qquad\forall v_{H_K^0}\in \mathcal{V}_{H_K^0};
\end{align*}
STEP 2: For $k \geq 0$, find $u_{H_K^{k+1}}\in \mathcal{V}_{H_K^{k+1}}$ such that
\begin{align*}
A_1(u_{H_K^k},u_{H_K^{k+1}},v_{H_K^{k+1}}) = 0\qquad\forall v_{H_K^{k+1}}\in \mathcal{V}_{H_K^{k+1}}.
\end{align*}
\end{algorithm}

Adding one-step Newton iteration in Algorithm \ref{algo3}, we obtain the following Algorithm \ref{algo4}.
\begin{algorithm}[H]
\caption{The ATG finite element algorithm with one-step Newton correction for mildly nonlinear problems}
\label{algo4}
STEP 1: Find $u_{H_K^0}\in \mathcal{V}_{H_K^0}$ such that
\begin{align*}
\widehat{A}(u_{H_K^0},v_{H_K^0}) = 0\qquad\forall v_{H_K^0}\in \mathcal{V}_{H_K^0};
\end{align*}
STEP 2: For $k\geq 0$, find $u_{H_K^{k+1}}\in \mathcal{V}_{H_K^{k+1}}$ such that
\begin{align*}
A_2(\tilde{u}_{H_K^{k+1}},u_{H_K^{k+1}},v_{H_K^{k+1}}) =& A_2(\tilde{u}_{H_K^{k+1}},\tilde{u}_{H_K^{k+1}},v_{H_K^{k+1}})\\
&-\hat{A}(\tilde{u}_{H_K^{k+1}},v_{H_K^{k+1}})\qquad\forall v_{H_K^{k+1}}\in \mathcal{V}_{H_K^{k+1}}, 
\end{align*}
where $\tilde{u}_{H_K^{k+1}}\in \mathcal{V}_{H_K^{k+1}}$ satisfies 
\begin{align*}
A_1(u_{H_K^k},\tilde{u}_{H_K^{k+1}},v_{H_K^{k+1}}) = 0\qquad\forall v_{H_K^{k+1}}\in \mathcal{V}_{H_K^{k+1}}. 
\end{align*}
\end{algorithm}

Next we present an {\em a posteriori} error estimator for Algorithm \ref{algo3}, and show that it is reliable and efficient. The convergence of the algorithm will also be addressed. 
 In this subsection, we assume the first order derivatives of $\widehat\alpha(\cdot,\cdot)$, $\widehat\beta(\cdot,\cdot)$, and $\widehat\gamma(\cdot,\cdot)$ are bounded.

\subsubsection{An {\em a posteriori} error estimator for Algorithm \ref{algo3}}\label{subsec4.1.1}
In this section, we introduce an {\em a posteriori} error estimator for the solution $\uk$ of  Algorithm \ref{algo3}. 
First, define the element residual and the edge jump by 
\begin{align}
\label{sec:re:residual-ele:m}
\R2K^k =& - \text{div} (\widehat\alpha(x,\up) \nabla \uk) + \widehat\beta(x,\up) \cdot \nabla \uk \\
&+ \widehat\gamma(x,\up) \qquad \forall \, K \in \T_k, \notag\\
\J2E^k =& \[ \widehat\alpha(x,\up) \nabla \uk \]_E \qquad \forall \, E \in \E^i_k. 
\end{align}

Now we define a global residual based {\em a posteriori} error estimator 
\begin{align}
\label{sec:re:est:1:m}
\e2(\uk,\T_k) &= \left( (\e2^k_R)^2 + (\e2^k_J)^2 \right)^{1/2}, \\
\label{sec:re:est:2:m}
\e2^k_R &= \left( \sum_{K \in \T_k} (\e2^k_{R,K})^2 \right)^{1/2}, \\
\label{sec:re:est:3:m}
\e2^k_J &= \left( \sum_{E \in \E^i_k} (\e2^k_{J,E})^2 \right)^{1/2}, 
\end{align}
where 
\begin{align}
(\e2^k_{R,K})^2 &= (\H^k)^2 \| \R2K^k \|^2_{L^2(K)} \qquad \forall \, K \in \T_k, \\
(\e2^k_{J,E})^2 &= H_E^k \| \J2E^k \|^2_{L^2(E)} \qquad \forall \, E \in \E^i_k. 
\end{align}

Similar to \eqref{osc_1}--\eqref{osc_3}, we can define three oscillation terms $\tilde{osc}^R(\uk, K)$, 
$\tilde{osc}^J(\uk, E)$, and $\tilde{osc}(\uk, \T_k)$ by 
replacing $R^k_K$ and $J^k_E$ with $\R2K^k$ and $\J2E^k$, respectively.  

\subsubsection{Reliability and efficiency for Algorithm \ref{algo3}} 
In this section, we derive a reliable upper bound and an efficient lower bound of the error $\| u - \uc \|_{H^1(\O)}$ for Algorithm \ref{algo3}. 

First, a reliable bound is given to show that the error of the ATG finite element Algorithm \ref{algo3} can be bounded by the error estimator up to higher-order terms.
\begin{theorem}\label{thm:re:reliable:m}
Let $u$ and $u_{H_K^{k+1}}$ be the solution of \eqref{eq20180219_5}--\eqref{eq20180219_6} and the adaptive two-grid finite element algorithm \ref{algo3}, respectively. Then there holds 
\begin{align*}
\| u - \uc \|_{H^1(\O)} \leq C \left( \e2(\uc, \T_{k+1}) + \| u - \uk \|_{L^2(\O)} + \| u - \uc \|_{L^2(\O)} \right). 
\end{align*}
\end{theorem}

\begin{proof}
It follows from \eqref{eq20190820_1} that
\begin{align}\label{re:1}
| u - \uc |^2_{H^1(\Omega)} \leq C (\widehat\alpha (x,u) \nabla (u - \uc), \nabla (u - \uc)). 
\end{align}
Therefore, it suffices to estimate $(\widehat\alpha(x,u) \nabla (u - \uc), \nabla v)$ for $v \in H^1_0(\O)$. 

Let $v^I$ be the Scott-Zhang interpolation \cite{scott1990finite} of $v$ on $\mathcal{V}_{H_K^{k+1}}$. From the weak formulation of \eqref{eq20190702_2}--\eqref{eq20190702_3} and Step 2 of Algorithm \ref{algo3}, we have for any $v \in H^1_0(\O)$, 
\begin{align}\label{re:2}
&(\widehat\alpha(x,u) \nabla (u - \uc), \nabla v) \\
=& (\widehat\alpha(x,u) \nabla u, \nabla v) - (\widehat\alpha(x,u) \nabla \uc, \nabla v) \notag \\
=& -(\widehat\beta(x,u) \cdot \nabla u + \widehat\gamma(x,u), v) - (\widehat\alpha(x,u) \nabla \uc, \nabla v) \notag \\
=& -( (\widehat\beta(x,u) - \widehat\beta(x,\uk) \cdot \nabla u , v) - (\widehat\beta(x,\uk) \cdot \nabla \uc, v - v^I) \notag \\
&\quad - (\widehat\beta(x,\uk) \cdot (\nabla u - \nabla \uc), v) - (\widehat\gamma(x,u) - \widehat\gamma(x,\uk) ,v) \notag \\
&\quad - (\widehat\gamma(x,\uk), v - v^I) - (\widehat\alpha(x,\uk) \nabla \uc, \nabla (v-v^I)) \notag \\
&\quad -( (\widehat\alpha(x,u) - \widehat\alpha(x,\uk)) \nabla \uc, \nabla v). \notag
\end{align} 

By integration by parts and properties (see Theorem 4.1 in \cite{scott1990finite}) of the Scott-Zhang interpolation, we have 
\begin{align}\label{re:3}
- (\widehat\gamma&(x,\uk), v - v^I) - (\widehat\alpha(x,\uk) \nabla \uc, \nabla (v-v^I)) \\
& \quad - (\widehat\beta(x,\uk) \cdot \nabla \uc, v - v^I) \notag \\
& = - \sum_{K \in \T_{k+1}} \int_K \R2K^{k+1} (v - v^I) \, dx - \sum_{E \in \E^i_{k+1}} \[ \widehat\alpha(x,\uk) \nabla \uc \] (v - v^I) \, ds \notag \\
& \leq C \tilde{\eta}(\uc, \T_{k+1}) \| v \|_{H^1(\O)}, \notag 
\end{align}
and 
\begin{align}\label{re:4}
- (\widehat\beta(x,\uk) \cdot (\nabla u - \nabla \uc), v) &= (u - \uc, \mbox{div} (\widehat\beta(x,\uk) v) ) \\
&\leq C \| u - \uc \|_{L^2(\O)} \|v\|_{H^1(\O)}. \notag 
\end{align}
The remaining terms on the right-hand side of \eqref{re:2} can be estimated by 
\begin{align}\label{re:5}
| ( (\widehat\beta(x,u) - \widehat\beta(x,\uk) \cdot \nabla u , v) | 
&\leq C \| u - \uk \|_{L^2(\O)} \|v\|_{H^1(\O)}, \\
| (\widehat\gamma(x,u) - \widehat\gamma(x,\uk) ,v) | &\leq C \| u - \uk \|_{L^2(\O)} \|v\|_{H^1(\O)}, \label{re:6} \\
| ( (\widehat\alpha(x,u) - \widehat\alpha(x,\uk)) \nabla \uc, \nabla v) | 
&\leq C \| u - \uk \|_{L^2(\O)} \|v\|_{H^1(\O)}. \label{re:7}
\end{align}

The constant $C$ here depends on $\widehat\alpha(x,u), \widehat\beta(x,u)$, and $\widehat\gamma(x,u)$. Finally, the theorem is proved by combining \eqref{re:2}--\eqref{re:7}. 
\end{proof}

For the {\em a posteriori} error estimates of the classical finite element algorithms, the oscillation terms are usually higher-order terms; however, it may not be true for 
Algorithm \ref{algo3} due to the low regularity of the numerical solution. A byproduct in this section is to give an upper bound of the oscillation terms, which plays a crucial role in proving the efficiency of the estimator of the adaptive two-grid algorithm. We will show that the oscillation terms are bounded by the summation of the errors and higher-order terms.
\begin{lemma}\label{lem20150623_11}
The oscillation term, which is defined in the end of Subsection \ref{subsec4.1.1}, can be bounded as below
\begin{align*}
\tilde{osc}(u_{H_K^{k+1}},\mathcal{T}_{k+1})&\leq C(\tilde{e}_1^{k+1}+\tilde{e}_2^{k+1}).
\end{align*}
where
\begin{align*}
\tilde{e}_1^{k+1}&:= \|\nabla (u - u_{H_K^{k+1}})\|_{L^2(\Omega)}+\bigg(\sum_{K\in\mathcal{T}_{k+1}}(H_K^{k+1})^2\|D^2u -D^2 u_{H_K^{k+1}}\|_{L^2(K)}^2\bigg)^\frac12\notag,\\
\tilde{e}_2^{k+1}&:=\|u - u_{H_K^k}\|_{L^2(\Omega)}+\bigg(\sum_{K\in\mathcal{T}_{k+1}}(H_K^{k+1})^2\|\nabla(u - u_{H_K^k})\|_{L^2(K)}^2\bigg)^\frac12\notag\\
& +\bigg(\sum_{K\in\mathcal{T}_{k+1}}(H_K^{k+1})^2\|\sigma - \bar{\sigma}\|_{L^2(K)}^2\bigg)^{\frac12}, \notag 
\end{align*}
and $\sigma$ and $\bar{\sigma}$ are defined in the beginning of the proof.
\end{lemma}

\begin{proof}
For the residual estimator, define $\hat{\sigma}^{k}$ and $\sigma$ by
\begin{align*}
\hat{\sigma}^k &:= -\widehat\alpha(x,u_{H_K^k})^T:D^2u_{H_K^{k+1}}-\mbox{div}(\widehat\alpha(x,u_{H_K^k})^T)\cdot\nabla u_{H_K^{k+1}}\\
&\qquad + \widehat\beta(x,u_{H_K^k})\cdot\nabla u_{H_K^{k+1}} + \widehat\gamma(x,u_{H_K^k}),\\
\sigma &:= -\widehat\alpha(x,u)^T:D^2u-\mbox{div}(\widehat\alpha(x,u)^T)\cdot\nabla u+ \widehat\beta(x,u)\cdot\nabla u + \widehat\gamma(x,u).
\end{align*}
Denote $\bar{\hat{\sigma}}^k$ and $\bar{\sigma}$ as the $L^2$ projections of 
$\hat{\sigma}^k$ and $\sigma$ to the piecewise $\mathcal{P}_{r-1}$ space, respectively, and then
\begin{align}\label{eq20150623_12}
&\|\hat{\sigma}^k - \bar{\hat{\sigma}}^k\|_{L^2(K)}\\
\leq&\|\hat{\sigma}^k - \bar{\sigma}\|_{L^2(K)}\notag\\
\leq&\|\hat{\sigma}^k - \sigma\|_{L^2(K)}+\|\sigma - \bar{\sigma}\|_{L^2(K)}\notag\\
\leq&C\|D^2u - D^2u_{H_K^{k+1}}\|_{L^2(K)}+C\|\nabla (u - u_{H_K^{k+1}})\|_{L^2(K)}\notag\\
&\quad+C\|\nabla (u - u_{H_K^k})\|_{L^2(K)}+C\|u - u_{H_K^k}\|_{L^2(K)}+\|\sigma - \bar{\sigma}\|_{L^2(K)},\notag
\end{align}
where the triangle inequality is used in the last inequality.

For the jump estimator, define $\hat{\delta}^{k}$ and $\delta$ on $E \in \E^i_{k+1}$ by
\begin{align*}
\hat{\delta}^{k} &:= [\![\widehat\alpha(x,u_{H_K^k})\nabla u_{H_K^{k+1}}]\!]_E,\\
\delta &:= [\![\widehat\alpha(x,u)\nabla u]\!]_E.
\end{align*}
Again, we denote $\bar{\hat{\delta}}^k$ and $\bar{\delta}$ as the piecewise $L^2 $ projections of $\hat{\delta}^k$ and $\delta$ to $\mathcal{P}_{r-1}$, respectively. 
Using the trace inequality with scaling and the triangle inequality, we have 
\begin{align}\label{add4}
&\|\hat{\delta}^k - \bar{\hat{\delta}}^k \|_{L^2(E)}\\
\leq & \| \hat{\delta}^k - \bar{{\delta}} \|_{L^2(E)} \notag\\
\leq &\| \hat{\delta}^k - {{\delta}} \|_{L^2(E)} + \| \delta  - \bar{{\delta}} \|_{L^2(E)} \notag \\
\leq &\| \[ \widehat\alpha(x,u_{H_K^k}) (\nabla \uc - \nabla u) \] \|_{L^2(E)} \notag \\
& + \| \[ (\widehat\alpha(x,u_{H_K^k})-\alpha(u))\nabla u \] \|_{L^2(E)}+ \| \delta  - \bar{{\delta}} \|_{L^2(E)} \notag \\
\leq& C \sum_{K \in \omega_E} \Big( (\H^{k+1})^{-\frac12} \|\nabla(u - u_{H_K^{k+1}})\|_{L^2(K)}+ (\H^{k+1})^{\frac12}\|D^2u - D^2u_{H_K^{k+1}}\|_{L^2(K)} \notag \\
& + (\H^{k+1})^{-\frac12}   \|u-u_{H_K^k}\|_{L^2(K)}+C (\H^{k+1})^{\frac12}   \|\nabla(u-u_{H_K^k})\|_{L^2(K)} \Big) \notag\\
&+ \| \delta  - \bar{{\delta}} \|_{L^2(E)}, \notag 
\end{align}
where $\omega_E$ is the set of elements in $\T_{k+1}$ containing $E$ as an edge. 

Note that $\| \delta  - \bar{{\delta}} \|_{L^2(E)} = 0$ on $E \in \E^i_{k+1}$ since $u \in H^2(\Omega)$. Then using \eqref{eq20150623_12}, \eqref{add4}, and the inverse inequality, we have
\begin{align}\label{eq20160425_3}
&\tilde{osc}(u_{H_K^{k+1}},\mathcal{T}_{k+1})\\
\leq &C\bigg(\sum_{K\in\mathcal{T}_{k+1}}(H_K^{k+1})^2\|D^2u -D^2 u_{H_K^{k+1}}\|_{L^2(K)}^2\bigg)^\frac12\notag\\
&+C\bigg(\sum_{K\in\mathcal{T}_{k+1}}(H_K^{k+1})^2\|\nabla(u - u_{H_K^{k+1}})\|_{L^2(K)}^2\bigg)^\frac12\notag\\
&+C\bigg(\sum_{K\in\mathcal{T}_{k+1}}(H_K^{k+1})^2\|\nabla(u - u_{H_K^k})\|_{L^2(K)}^2\bigg)^\frac12+C\|u - u_{H_K^k}\|_{L^2(\Omega)}\notag\\
&+C\bigg(\sum_{K\in\mathcal{T}_{k+1}}(H_K^{k+1})^2\|\sigma - \bar{\sigma}\|_{L^2(K)}^2\bigg)^{\frac12}+C\|\nabla(u - u_{H_K^{k+1}})\|_{L^2(\Omega)}\notag\\
&+C\bigg(\sum_{E\in\mathcal{E}_{k+1}}H_K^{k+1}\|\delta  - \bar{{\delta}} \|_{L^2(E)}^2\bigg)^{\frac12}\notag\\
\leq &C\|\nabla (u - u_{H_K^{k+1}})\|_{L^2(\Omega)}+C\bigg(\sum_{K\in\mathcal{T}_{k+1}}(H_K^{k+1})^2\|D^2u -D^2 u_{H_K^{k+1}}\|_{L^2(K)}^2\bigg)^\frac12\notag\\
&+C\|u - u_{H_K^k}\|_{L^2(\Omega)}+C\bigg(\sum_{K\in\mathcal{T}_{k+1}}(H_K^{k+1})^2\|\nabla(u - u_{H_K^k})\|_{L^2(K)}^2\bigg)^\frac12\notag\\
&+C\bigg(\sum_{K\in\mathcal{T}_{k+1}}(H_K^{k+1})^2\|\sigma - \bar{\sigma}\|_{L^2(K)}^2\bigg)^{\frac12} \notag \\ 
=& \tilde{e}_1^{k+1}+\tilde{e}_2^{k+1}\notag,
\end{align}
where $\tilde{e}_1^{k+1}, \tilde{e}_2^{k+1}$ denote the error term (the first two terms) and the higher-order terms (the last three terms) respectively.
\end{proof}

Now we are ready to prove the global efficiency of the error estimator up to higher-order terms. 
\begin{theorem}\label{thm:re:efficiency:m}
Let $u$ and $u_{H_K^{k+1}}$ be the solutions of \eqref{eq20180219_5}--\eqref{eq20180219_6} and the adaptive two-grid finite element algorithm \ref{algo3}, respectively. Then there holds  
\begin{align*}
\e2(\uc, \T_{k+1}) \leq C \left( \tilde{e}_1^{k+1} +\tilde{e}_2^{k+1} \right). 
\end{align*}
\end{theorem}

\begin{proof}
We will establish the following local efficient estimates 
\begin{align}\label{re:e-est:1:m}
(\e2_{R,K}^{k+1})^2 &\leq C \big( \| u - \uc \|_{H^1(K)}^2 + \| u - \uk \|^2_{L^2(K)} \\
& \quad +  ( \tilde{osc}^{R}(\uc,K) )^2 \big) \qquad \forall \, K \in \T_{k+1}, \notag \\
(\e2_{J,E}^{k+1})^2 &\leq C \big( \| u - \uc \|_{H^1(\omega_E)}^2 + \| u - \uk \|^2_{L^2(\omega_K)} \label{re:e-est:2:m} \\
& \quad + \sum_{K \in \omega_E} (\tilde{osc}^R(\uc,K))^2 + (\tilde{osc}^J(\uc,E))^2 \big) \qquad \forall \, E \in \E^i_{k+1}. \notag 
\end{align}
Then the theorem follows from \eqref{re:e-est:1:m}--\eqref{re:e-est:2:m} and Lemma \ref{lem20150623_11}. 

For $K \in \T_{k+1}$, let $\varphi_K$ be the element bubble function \cite{verfurth2013posteriori}, and then we have 
\begin{align}\label{re:e-est:3:m}
(\bar{R}^{k+1}_K, \varphi_K \bar{R}^{k+1}_K)_K &= (\bar{R}^{k+1}_K - \R2K^{k+1}, \varphi_K \bar{R}^{k+1}_K)_K + (\R2K^{k+1}, \varphi_K \bar{R}^{k+1}_K)_K,  
\end{align}
where we used $(\cdot, \cdot)_K$ to denote the $L^2$ inner product in $L^2(K)$. 
By the weak formulation of \eqref{eq20190702_2}--\eqref{eq20190702_3}, integration by parts, and Proposition 1.4 in \cite{verfurth2013posteriori}, the second term on the right-hand side of \eqref{re:e-est:3:m} can be estimated by 
\begin{align}\label{re:e-est:4:m}
(\R2K^{k+1}, \varphi_K \bar{R}^{k+1}_K)_K &= (\widehat\alpha(x,\uk) \nabla \uc - \widehat\alpha(x,u) \nabla u, \nabla (\varphi_K \bar{R}^{k+1}_K))_K \\
&\quad + (\widehat\beta(x,\uk) \cdot \nabla \uc - \widehat\beta(x,u) \cdot \nabla u, \varphi_K \bar{R}^{k+1}_K))_K \notag \\
&\quad + (\widehat\gamma(x,\uk) - \widehat\gamma(x,u), \varphi_K \bar{R}^{k+1}_K)_K \notag \\
\leq C (H^{k+1}_K)^{-1}& \big( \| u - \uc \|_{L^2(K)} + \| u - \uk \|_{L^2(K)} \big) \| \bar{R}^{k+1}_K \|_{L^2(K)}. \notag 
\end{align}
Note that 
\begin{align}\label{re:e-est:5:m}
&(\bar{R}^{k+1}_K - \R2K^{k+1}, \varphi_K \bar{R}^{k+1}_K)_K \leq C \| \bar{R}^{k+1}_K - \R2K^{k+1} \|_{L^2(K)} \| \bar{R}^{k+1}_K \|_{L^2(K)}, \\
&\| \bar{R}^{k+1}_K \|^2_{L^2(K)} = (\bar{R}^{k+1}_K, \bar{R}^{k+1}_K)_K \leq C (\bar{R}^{k+1}_K, \varphi_K \bar{R}^{k+1}_K)_K, \label{re:e-est:6:m}
\end{align}
and then the estimate \eqref{re:e-est:1:m} follows from \eqref{re:e-est:3:m}--\eqref{re:e-est:6:m} and the triangle inequality.  

Next, for any $E \in \E_{k+1}$, using the properties of the edge bubble function $\varphi_E$ (see Proposition 1.4 in \cite{verfurth2013posteriori}) and the integration by parts, we have 
\begin{align}\label{re:e-est:7:m}
(\tilde{J}_E^{k+1}, \varphi_E \bar{J}^{k+1}_E)_E &= (\widehat\alpha(x,\uk) \nabla \uc - \widehat\alpha(x,u) \nabla u, \nabla (\varphi_E \bar{J}_E^{k+1}) )_{\omega_E} \notag \\
&\quad + (\widehat\beta(x,\uk) \cdot \nabla \uc - \widehat\beta(x,u) \cdot \nabla u, \varphi_E \bar{J}_E^{k+1} )_{\omega_E} \\
&\quad + (\widehat\gamma(x,\uk) - \widehat\gamma(x,u), \varphi_E \bar{J}_E^{k+1})_{\omega_E} - (\R2K^{k+1}, \varphi_E \bar{J}_E^{k+1})_{\omega_E} \notag \\
\leq& C (H^{k+1}_E)^{-\frac12} \big( \| u - \uc \|_{H^1(\omega_E)} + \| u - \uk \|_{L^2(\omega_E)} \big) \| \bar{J}_E^{k+1} \|_{L^2(E)} \notag \\
&\quad + C (H^{k+1}_E)^{\frac12} \| \tilde{R}_K^{k+1} \|_{L^2(\omega_E)} \| \bar{J}_E^{k+1} \|_{L^2(E)}. \notag 
\end{align}
By applying the triangle inequality, \eqref{re:e-est:1:m}, and the fact that 
\begin{align}\label{re:e-est:8:m}
(\bar{J}_E^{k+1}, \bar{J}_E^{k+1})_E &\leq C (\bar{J}_E^{k+1}, \varphi_E \bar{J}_E^{k+1})_E \\
&= C (\tilde{J}_E^{k+1}, \varphi_E \bar{J}_E^{k+1})_E + C (\bar{J}_E^{k+1} - \tilde{J}^{k+1}_E, \varphi_E \bar{J}_E^{k+1})_E, \notag 
\end{align}
we complete the proof of \eqref{re:e-est:2:m}. 
\end{proof}

\subsubsection{Convergence of Algorithm~\ref{algo3}}\label{sec:conv:algo3}
In this section we discuss the convergence of the adaptive two-grid finite element Algorithm~\ref{algo3}, assuming the bulk criterion \eqref{20150623_7_add} holds. The following lemma and Lemma \ref{lem:conv1} are useful to prove the error reduction property.
\begin{lemma}\label{lem:conv2}
Assume the bulk criterion \eqref{20150623_7_add} holds, and then for any $\epsilon>0$, there exists $\beta_1(\epsilon) > 0$ and $C_\epsilon > 0$ depending on $\epsilon$ such that 
\begin{align}\label{lem:conv2:est}
\e2^2(u_{H^{k+1}_K},\mathcal{T}_{k+1})\leq&(1+\epsilon)(1-\rho\theta)\e2^2(u_{H_K^k},\mathcal{T}_{k}) \\
&+\frac{C}{\beta_1(\epsilon)} |\| u_{H^{k+1}_K}-u_{H_K^k}) |\|_2^2+C_\epsilon\|u_{H^{k}_K}-u_{H_K^{k-1}}\|_{L^2(\Omega)}^2.  \notag
\end{align}
\end{lemma}

\begin{proof}
From \eqref{sec:re:est:1:m}, we have 
\begin{align}\label{lem:conv2:p1}
&|\e2(u_{H^{k+1}_K},\mathcal{T}_{k+1})-\e2(u_{H_K^k},\mathcal{T}_{k+1})|\\
=& \bigg| \bigg(\sum_{K\in\mathcal{T}_{k+1}}(H^{k+1}_K)^2\|-\mbox{div}(\widehat\alpha(x,\uk)\nabla u_{H^{k+1}_K}) + \widehat\beta(x,\uk)\cdot\nabla u_{H_K^{k+1}}  \notag\\
&\qquad+ \widehat\gamma(x,\uk)\|_{L^2(K)}^2+ \sum_{E\in\mathcal{E}_{k+1}}H^{k+1}_E\bigl\|[\![\widehat\alpha(x,\uk)\nabla u_{H^{k+1}_K}]\!]\bigr\|_{L^2(E)}^2 \bigg)^{1\slash2}\notag\\
&-\bigg(\sum_{K\in\mathcal{T}_{k+1}}(H^{k+1}_K)^2\|-\mbox{div}(\widehat\alpha(x,\up)\nabla u_{H_K^{k}}) + \widehat\beta(x,\up)\cdot\nabla u_{H_K^{k}}  \notag\\
&\qquad+ \widehat\gamma(x,\up)\|_{L^2(K)}^2+ \sum_{E\in\mathcal{E}_{k+1}}H^{k+1}_E\bigl\|[\![\widehat\alpha(x,\up)\nabla u_{H_K^k}]\!]\bigr\|_{L^2(E)}^2 \bigg)^{1\slash2} \bigg| \notag\\
\leq& \bigg( \sum_{K\in\mathcal{T}_{k+1}}(H^{k+1}_K)^2\|-\mbox{div}(\widehat\alpha(x,\uk)\nabla u_{H^{k+1}_K}) + \mbox{div}(\widehat\alpha(x,\up) \nabla \uk) \|^2_{L^2(K)} \notag \\
&\qquad + \sum_{K\in\mathcal{T}_{k+1}}(H^{k+1}_K)^2 \|\widehat\beta(x,\uk) \cdot \nabla \uc - \widehat\beta(x,\up) \cdot \nabla \uk \|^2_{L^2(K)} \notag \\
&\qquad + \sum_{K\in\mathcal{T}_{k+1}}(H^{k+1}_K)^2 \| \widehat\gamma(x,\uk) - \widehat\gamma(x,\up)\|^2_{L^2(K)} \notag \\
&\qquad + \sum_{E\in\mathcal{E}_{k+1}}H^{k+1}_E\bigl\|[\![\widehat\alpha(x,\uk)\nabla u_{{H^{k+1}_K}} - \widehat\alpha(x,\up) \nabla \uk]\!]\bigr\|_{L^2(E)}^2 \bigg)^{1/2}. \notag 
\end{align}

Now we estimate the first term on the right-hand side of \eqref{lem:conv2:p1}: 
\begin{align}\label{lem:conv2:p2}
\sum_{K\in\mathcal{T}_{k+1}}&(H^{k+1}_K)^2\|-\mbox{div}(\widehat\alpha(x,\uk)\nabla u_{H^{k+1}_K}) + \mbox{div}(\widehat\alpha(x,\up) \nabla \uk) \|^2_{L^2(K)} \\
&\leq C \sum_{K\in\mathcal{T}_{k+1}} (H^{k+1}_K)^2 \| \mbox{div} (\widehat\alpha(x,\uk)^T) \cdot \nabla (\uc - \uk) \|^2_{L^2(K)} \notag \\
& + C \sum_{K\in\mathcal{T}_{k+1}} (H^{k+1}_K)^2 \| \mbox{div}(\widehat\alpha(x,\uk)^T - \widehat\alpha(x,\up)^T) \cdot \nabla \uk \|^2_{L^2(K)}  \notag \\
& + C \sum_{K\in\mathcal{T}_{k+1}} (H^{k+1}_K)^2 \|\widehat \alpha(x,\uk)^T\cdot(D^2 \uc - D^2 \uk) \|^2_{L^2(K)} \notag \\
& + C \sum_{K\in\mathcal{T}_{k+1}} (H^{k+1}_K)^2 \| (\widehat\alpha(x,\uk)^T -\widehat\alpha(x,\up)^T)\cdot D^2 \uk \|^2_{L^2(K)} \notag \\
&\leq C\|\nabla(u_{H^{k+1}_K}-u_{H_K^k})\|^2_{L^2(\Omega)} + C\|u_{H^{k}_K}-u_{H_K^{k-1}}\|^2_{L^2(\Omega)},\notag 
\end{align}
where the inverse inequality is used. 
 The second and third terms on the right-hand side of \eqref{lem:conv2:p1} can be bounded similarly. 
By the trace theorem with scaling and the inverse inequality, we obtain the bound for the last term 
\begin{align}\label{lem:conv2:p3}
\sum_{E\in\mathcal{E}_{k+1}}&H^{k+1}_E\bigl\|[\![\widehat\alpha(x,\uk)\nabla u_{{H^{k+1}_K}} - \widehat\alpha(x,\up) \nabla \uk]\!]\bigr\|_{L^2(E)}^2 \\
&\leq C \Big( \sum_{E\in\mathcal{E}_{k+1}}H^{k+1}_E\bigl\|[\![\widehat\alpha(x,\uk) (\nabla u_{{H^{k+1}_K}} - \nabla \uk)]\!]\bigr\|_{L^2(E)}^2 \notag \\
&\qquad + \sum_{E\in\mathcal{E}_{k+1}}H^{k+1}_E\bigl\|[\![ (\widehat\alpha(x,\uk) - \widehat\alpha(x,\up) \nabla \uk]\!]\bigr\|_{L^2(E)}^2 \Big) \notag \\
&\leq C \| \nabla (\uc - \uk) \|^2_{L^2(\O)} + C\|u_{H^{k}_K}-u_{H_K^{k-1}}\|^2_{L^2(\Omega)}. \notag 
\end{align}

It now follows from \eqref{lem:conv2:p1}--\eqref{lem:conv2:p3}, Young's inequality, and the bulk criterion \eqref{20150623_7_add} that 
\begin{align}\label{20150623_7}
\eta^2(u_{H_K^k},\mathcal{T}_{k}\backslash\mathcal{T}_{k+1})\geq\theta \eta^2(u_{H_K^k},\mathcal{T}_{k}),
\end{align}

and thus 
\begin{align}\label{lem:conv2:p4}
\e2^2(\uc,\T_{k+1}) &\leq (1+\epsilon) \e2^2(\uk,\T_{k+1}) + \frac{C}{\beta_1(\epsilon)} \| \nabla (\uc - \uk) \|^2_{L^2(\O)} \\
&\qquad + C_\epsilon \| \uk - \up \|^2_{L^2(\O)} \notag \\
&\leq (1+\epsilon)(1 - \rho \theta) \e2^2(\uk,\T_{k}) + \frac{C}{\beta_1(\epsilon)} |\| \uc - \uk \||_2^2 \notag \\
&\qquad + C_\epsilon \| \uk - \up \|^2_{L^2(\O)}, \notag 
\end{align}
where $\beta_1(\epsilon) > 0$ and $C_\epsilon > 0$ depend on $\epsilon$. We complete the proof. 
\end{proof}

Define $|\|u-u_{H^{k+1}_K}\||_2^2 + C\eta^2(u_{H^{k+1}_K},\mathcal{T}_{k+1})$ to be  the quasi-error on the mesh $\T_{k+1}$. 
The following result shows that the quasi-error decreases with respect to the number of mesh bisections up to some $L^2$ norms of the errors, which are higher-order terms on uniform/adaptive meshes. 
\begin{theorem} (Error reduction)\label{thm:conv2:err-red}
There exist $\tilde{\rho} > 0$ and $0 < \tilde{\zeta} < 1$ such that 
\begin{align*}
|\|u-u_{H^{k+1}_K}\||_2^2 + & \tilde{\rho} \e2^2(u_{H^{k+1}_K},\mathcal{T}_{k+1})\leq\tilde{\zeta}\bigl(|\|u-u_{H_K^k}\||_2^2 + \tilde{\rho} \e2^2(u_{H_K^k},\mathcal{T}_{k})\bigr)\\
&\quad + C (\| u - \uc \|^2_{L^2(\O)} + \|u-u_{H_K^k}\|_{L^2(\Omega)}^2+\|u-u_{H_K^{k-1}}\|_{L^2(\Omega)}^2). 
\end{align*}
\end{theorem}

\begin{proof}
From \eqref{norm:conv2}, we have 
\begin{align}\label{thm:conv2:err-red:p1}
|\|u-u_{H^{k+1}_K}\||_2^2 &= |\|u-u_{H_K^k}\||_2^2 - |\|u_{H^{k+1}_K}-u_{H_K^k}\||_2^2 \\
& \qquad-2(\widehat\alpha(x,u)\nabla(u-u_{H^{k+1}_K}),\nabla(u_{H^{k+1}_K}-u_{H_K^k})).\notag
\end{align}
By Algorithm~\ref{algo3}, integration by parts, 
and Poincar\'{e}'s inequality, we have 
\begin{align}\label{thm:conv2:err-red:p2}
&- (\widehat\alpha(x,u) \nabla(u-u_{H^{k+1}_K}),\nabla(u_{H^{k+1}_K}-u_{H_K^k})) \\
= \ &  (\widehat\beta(x,u)\cdot\nabla u + \widehat\gamma(x,u), \uc - \uk)  - (\widehat\beta(x,\uk)\cdot\nabla \uc + \widehat\gamma(x,\uk), \notag \\
&  \uc - \uk)+ \left( (\widehat\alpha(x,u) - \alpha(\uk)) \nabla \uc, \nabla(\uc - \uk) \right) \notag \\
\leq \ & \epsilon_1 |\| u - \uc \||_2^2 + \epsilon_1 |\| u - \uk \||_2^2 + C(\epsilon_1) \| u - \uk \|_{L^2(\O)}^2 \notag \\
& + C(\epsilon_1) \| u - \uc \|_{L^2(\O)}^2, \notag 
\end{align}
where $C(\epsilon_1) > 0$ depends on $\epsilon_1 \in (0, 1)$, and $C(\epsilon_1)$ will be determined later. 

Combining \eqref{thm:conv2:err-red:p1} and \eqref{thm:conv2:err-red:p2}, we have
\begin{align}\label{thm:conv2:err-red:p3}
(1 - \epsilon_1) |\| u - \uc \||_2^2 &\leq (1 + \epsilon_1) |\| u - \uk \||_2^2 - \| \uc - \uk \|^2_{H^1(\O)} \\
&\quad + C(\epsilon_1) \| u - \uk \|^2_{L^2(\O)} + C(\epsilon_1) \| u - \uc \|_{L^2(\O)}^2. \notag 
\end{align}

For convenience, we denote $\beta_2(\epsilon) = (1+\epsilon)(1-\rho\theta)$ and $\gamma_2 = \frac{\beta_1(\epsilon)}{C} > 0$. 
Next, we combine  \eqref{lem:conv2:est} and \eqref{thm:conv2:err-red:p3} to get 
\begin{align}\label{thm:conv2:err-red:p4}
&(1-\epsilon_1)|\|u-u_{H^{k+1}_K}\||_2^2 + \gamma_2 \e2^2(u_{H^{k+1}_K},\mathcal{T}_{k+1})\\
&\qquad\leq (1+\epsilon_1) |\|u-u_{H_K^k}\||_2^2 + \gamma_2 \beta_2(\epsilon) \e2^2(u_{H_K^k},\mathcal{T}_{k})\notag\\
&\qquad\qquad+ \big(\frac{C \gamma_2}{\beta_1(\epsilon)} - 1\big) \|u_{H^{k+1}_K}-u_{H_K^k}\|_{L^2(\Omega)}^2 + C \gamma_2 \| \uk - \up \|^2_{L^2(\O)} \notag\\
&\qquad\qquad + C(\epsilon_1) \| u - \uk \|^2_{L^2(\O)} + C(\epsilon_1) \| u - \uc \|^2_{L^2(\O)} \notag \\ 
&\qquad = \tilde{\zeta} \left[ (1-\epsilon_1) |\|u-u_{H_K^k}\||_2^2 + \gamma_2 \e2^2(u_{H_K^k},\mathcal{T}_{k}) \right] \notag \\
&\qquad\qquad + \left[ (1+\epsilon_1) - \tilde{\zeta} (1-\epsilon_1) \right] \| u - u_{H^{k}_K} \|_{L^2(\Omega)}^2 \notag \\
&\qquad\qquad + \gamma_2 (\beta_2(\epsilon) - \tilde{\zeta}) \e2^2(u_{H_K^k},\mathcal{T}_{k}) + C \gamma_2 \| \uk - \up \|^2_{L^2(\O)} \notag \\ 
&\qquad\qquad  + C(\epsilon_1) \| u - \uk \|^2_{L^2(\O)} + C(\epsilon_1) \| u - \uc \|^2_{L^2(\O)}, \notag
\end{align}
for some $\tilde{\zeta} > 0$ that will be specified later. Applying the reliability of the estimator in Lemma \ref{thm:re:reliable:m} yields
\begin{align}\label{thm:conv2:err-red:p5}
|\| u - \uk \||_2^2 \leq C_{Rel} \left( \e2^2(\uk, \T_k) + \| u - \up \|^2_{L^2(\O)} + \| u - \uk \|^2_{L^2(\O)} \right), 
\end{align}
where $C_{Rel}$ is the constant in Lemma \ref{thm:re:reliable:m}, 
and inserting this into the right-hand side of \eqref{thm:conv2:err-red:p4}, we thus have 
\begin{align}\label{thm:conv2:err-red:p6}
(1-\epsilon_1) &|\|u-u_{H^{k+1}_K}\||_2^2 + \gamma_2 \e2^2(u_{H^{k+1}_K},\mathcal{T}_{k+1}) \notag \\
&\leq \tilde{\zeta} \left[ (1-\epsilon_1) |\|u-u_{H_K^k}\||_2^2 + \gamma_2 \e2^2(u_{H_K^k},\mathcal{T}_{k}) \right] \\
&\quad + \big[ [(1+\epsilon_1) - \tilde{\zeta}(1-\epsilon_1)] C_{Rel} + \gamma_2 (\beta_2(\epsilon) - \tilde{\zeta}) \big] \e2^2(\uk, \T_k) \notag \\ 
&\quad + C \| u - \uk \|^2_{L^2(\O)} + C \| u - \up \|^2_{L^2(\O)} + C \| u - \uc \|^2_{L^2(\O)}. \notag  
\end{align}

Choosing $\epsilon$ small enough such that $0 < \beta_2(\epsilon) < 1$, we set 
\begin{align}\label{thm:conv2:err-red:p7}
\epsilon_1 &= \min \left\{ 1, \frac{\gamma_2(1-\beta_2(\epsilon))}{2 C_{Rel} } \right\}.
\end{align}
Furthermore, by choosing 
\begin{align}\label{thm:conv2:err-red:p8}
\tilde{\zeta} &= \frac{(1+\epsilon_1)C_{Rel} + \gamma_2 \beta_2(\epsilon)}{C_{Rel} (1-\epsilon_1) + \gamma_2}, 
\end{align}
we have $0 < \tilde{\zeta} < 1$ and 
the second term on the right-hand side of \eqref{thm:conv2:err-red:p6} vanishes. Finally, a re-scaling with the choice of $\tilde{\rho} = \frac{\gamma_2}{1-\epsilon_1}$ in \eqref{thm:conv2:err-red:p6} completes the proof.  
\end{proof}

The following convergence result is an immediate consequence of Theorem \ref{thm:conv2:err-red}. It shows that the error decreases to zero up to higher-order terms as meshes bisect. 
\begin{theorem} (Convergence)\label{thm:conv4.1}
Assume the bulk criterion \eqref{20150623_7_add} holds. With the choices of $\tilde{\zeta} \in (0, 1)$ and $\tilde{\rho} > 0$ in Theorem~\ref{thm:conv2:err-red},  
there holds 
\begin{align}\label{thm:conv4.1:est}
&|\|u-u_{H^{k+1}_K}\||_2^2 + \tilde{\rho} \e2^2(u_{H^{k+1}_K},\mathcal{T}_{k+1})\\
&\qquad\leq\tilde{\zeta}^{k+1}\bigl(|\|u-u_{H_K^0}\||_2^2+ \tilde{\rho} \e2^2(u_{H_K^0},\mathcal{T}_0)\bigr) \notag \\
&\qquad\quad+C \sum_{i=0}^{k}\|u-u_{H_K^{k+1-i}}\|_{L^2(\Omega)}^2\tilde{\zeta}^i + C \sum_{i=0}^{k}\|u-u_{H_K^{k-i}}\|_{L^2(\Omega)}^2\tilde{\zeta}^i\notag \\
&\qquad\quad + C \sum_{i=0}^{k-1}\|u-u_{H_K^{k-1-i}}\|_{L^2(\Omega)}^2\tilde{\zeta}^i. \notag
\end{align}
\end{theorem}

\begin{remark}\label{rem:addtest}
The last three terms on the right-hand side of \eqref{thm:conv4.1:est} appear to be higher-order terms on adaptive refinements (See Section \ref{sec-6}). 
In particular, we observe terms like $\| u - u_{H^k_K} \|_{L^2(\Omega)} \tilde{\zeta}^0$ and 
$\| u - u_{H^0_K} \|_{L^2(\Omega)} \tilde{\zeta}^k$, where the first term is considered to be small on the $k$-th mesh and the second term depends on the $L^2$ error on the initial mesh. Therefore, a sufficiently small initial mesh will be needed in order to achieve a more accurate approximation on the $k$-th mesh. 
\end{remark}

\subsection{General nonlinear PDEs}
The Newton method is employed in this section to solve the general nonlinear PDEs. One step of Newton iteration is used in Algorithm \ref{algo5}. 
\begin{algorithm}[H]
\caption{The ATG finite element algorithm with one-step Newton correction for general nonlinear problems}
\label{algo5}
STEP 1: Find $u_{H_K^0}\in \mathcal{V}_{H_K^0}$ such that
\begin{align*}
A(u_{H_K^0},v_{H_K^0}) = 0\qquad\forall v_{H_K^0}\in \mathcal{V}_{H_K^0};
\end{align*}
STEP 2: For $k \geq 0$, find $u_{H_K^{k+1}}\in \mathcal{V}_{H_K^{k+1}}$ such that
\begin{align*}
A_2(u_{H_K^k},u_{H_K^{k+1}},v_{H_K^{k+1}}) = A_2(u_{H_K^k},u_{H_K^k},v_{H_K^{k+1}})-A(u_{H_K^k},v_{H_K^{k+1}})\quad\forall v_{H_K^{k+1}}\in \mathcal{V}_{H_K^{k+1}}.
\end{align*}
\end{algorithm}

Corresponding to Algorithm \ref{algo5}, two steps of Newton iteration are used in the Algorithm \ref{algo7}.
\begin{algorithm}[H]
\caption{The ATG finite element algorithm with two-step Newton corrections for general nonlinear problems}
\label{algo7}
STEP 1: Find $u_{H_K^0}\in \mathcal{V}_{H_K^0}$ such that
\begin{align*}
A(u_{H_K^0},v_{H_K^0}) = 0\qquad\forall v_{H_K^0}\in \mathcal{V}_{H_K^0};
\end{align*}
STEP 2: For $k \geq 0$, find $u_{H_K^{k+1}}\in \mathcal{V}_{H_K^{k+1}}$ such that
\begin{align*}
A_2(\tilde{u}_{H_K^{k+1}},u_{H_K^{k+1}},v_{H_K^{k+1}}) =& A_2(\tilde{u}_{H_K^{k+1}},\tilde{u}_{H_K^{k+1}},v_{H_K^{k+1}})\\
&-A(\tilde{u}_{H_K^{k+1}},v_{H_K^{k+1}})\qquad\forall v_{H_K^{k+1}}\in \mathcal{V}_{H_K^{k+1}}, 
\end{align*}
where $\tilde{u}_{H_K^{k+1}}\in \mathcal{V}_{H_K^{k+1}}$ satisfies 
\begin{align*}
A_2(u_{H_K^{k}},\tilde{u}_{H_K^{k+1}},v_{H_K^{k+1}}) = A_2(u_{H_K^{k}},u_{H_K^{k}},v_{H_K^{k+1}})-A(u_{H_K^{k}},v_{H_K^{k+1}})\quad\forall v_{H_K^{k+1}}\in \mathcal{V}_{H_K^{k+1}}. 
\end{align*}
\end{algorithm}

\section{Numerical experiment}\label{sec-6}
In this section, some numerical experiments will be given to test the proposed  algorithms with linear finite elements. To present the numerical results clearly, we explain some of the proposed algorithms. Consider Figure \ref{fig1} as an example: the left graph is the the initial mesh and the right one is the mesh after one step of bisection.
\begin{figure}[H]
\centering
\includegraphics[height=1.6in,width=1.7in]{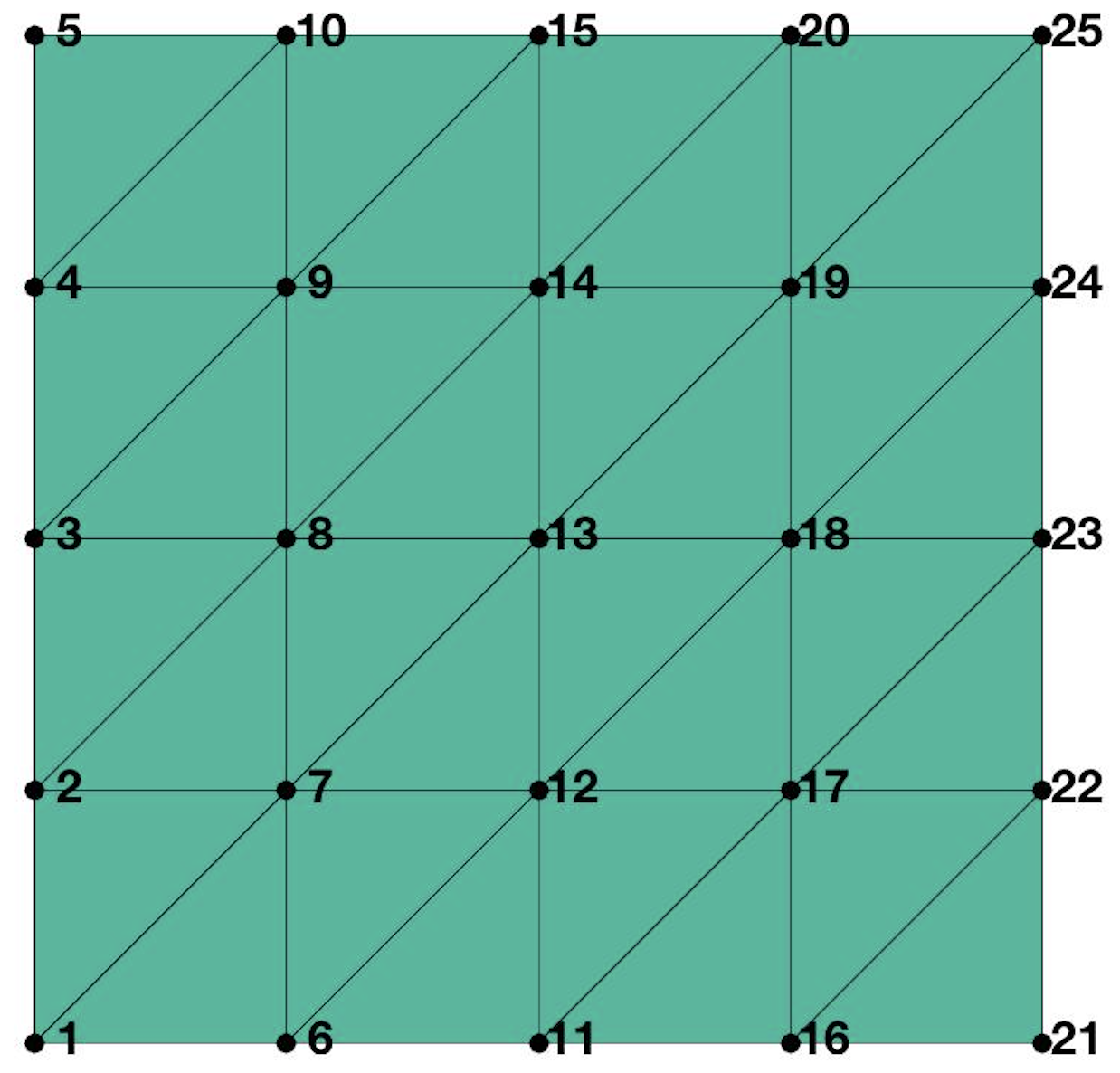}
\includegraphics[height=1.59in,width=1.7in]{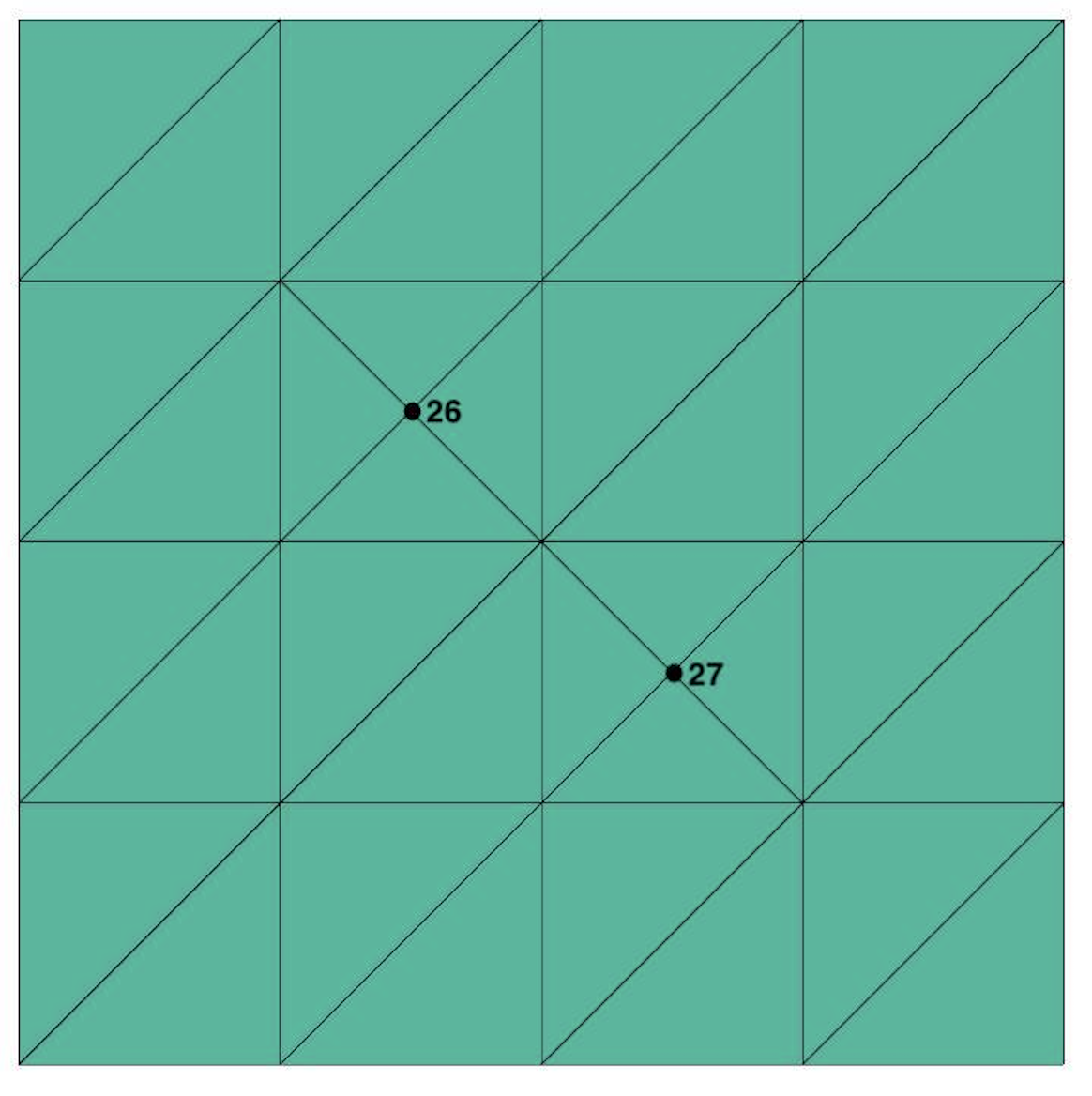}
\caption{The initial mesh (left) and the mesh after one step of bisection (right).}\label{fig1}
\end{figure}

The following steps help to explain the implementation of the algorithms \ref{algo3}-\ref{algo7}. We consider the first bisection, and the next bisections follow the similar way. 
\begin{enumerate}
\item In Step 1 of the Algorithms \ref{algo3}-\ref{algo7}, assume the numerical solution on the initial meshes (left in Figure \ref{fig1}) has been obtained, i.e., $[u_1^0\quad u_2^0\quad\cdots\quad u_{25}^0]^T$. Denote the numerical solution after the bisection as $[u_1^1\quad u_2^1\quad\cdots\quad u_{27}^1]^T$. The error estimator is computed to determine how to refine or coarsen meshes;
\item In Step 2 of the Algorithms \ref{algo3}-\ref{algo7}, mesh points $26$ and $27$ are added, and $u_{26}^0=\frac{u_9^0+u_{13}^0}{2}$ and $u_{27}^0=\frac{u_{13}^0+u_{17}^0}{2}$ are used as the coefficient of the Newton iterations. This is very efficient since only the newly-added mesh points need to be interpolated. Then $[u_1^1\quad u_2^1\quad\cdots\quad u_{27}^1]^T$ are computed based on $[u_1^0\quad u_2^0\quad\cdots\quad u_{27}^0]^T$.
\end{enumerate}

In Test 1, we compare the Algorithm \ref{algo3} to the regular adaptive method (see Algorithm \ref{algo3:regular}) as well as the two-grid algorithm \cite{xu1996two}. Besides, the last three terms on the right-hand side of \eqref{thm:conv4.1:est} are verified to be very small ($10^{-3}$ magnitude) compared to $H^1$ error ($10^{-1}$ magnitude) as tested below (see Remark \ref{rem:addtest}). In Test 2, Algorithms \ref{algo3} and \ref{algo5} are compared based on the convergence rates and higher-order terms. In Test 3, Algorithms \ref{algo3} and \ref{algo3:regular}) are compared for a semi-linear PDE with a $tanh$ profile solution.

\paragraph{Test 1} In this test, we choose the domain to be $\Omega=[-1,1]\times[-1,1]$ and the exact solution to be $u=\sin\pi x\sin\pi y$, and consider the following semi-linear PDEs (see \cite{feng2014analysis, xu2016convex} for more numerical examples of semi-linear PDEs)
\begin{alignat}{2}
-\Delta u + u^5 &= 2\pi^2\sin\pi x\sin\pi y+\sin^5\pi x\sin^5\pi y\qquad &&\text{in}\ \Omega,\label{eq20190528_1}\\
u &= 0\qquad &&\text{on}\ \partial\Omega\label{eq20190528_2}.
\end{alignat}

We compare Algorithm \ref{algo3} to the regular adaptive finite element algorithm (see Algorithm \ref{algo3:regular}), where we solve $\uc \in \mathcal{V}_{H^{k+1}}$ from the original nonlinear equation. Notice that the error estimator in Algorithm \ref{algo3} is generated from 
\begin{align*}
\R2K^k &=  (\up)^5-2\pi^2\sin\pi x\sin\pi y-\sin^5\pi x\sin^5\pi y \qquad \forall \, K \in \T_k, \\
\J2E^k &= \[\nabla \uk \]_E \qquad \forall \, E \in \E^i_k,
\end{align*}
where $\uk \in \T_k$ is the solution to Algorithm \ref{algo3}. For the regular adaptive finite element algorithm, the following element residual and the edge jump are used 
\begin{align*}
\R2K^k &=  (\uk)^5-2\pi^2\sin\pi x\sin\pi y-\sin^5\pi x\sin^5\pi y \qquad \forall \, K \in \T_k, \\
\J2E^k &= \[\nabla \uk \]_E \qquad \forall \, E \in \E^i_k, 
\end{align*}
where $\uk \in \T_k$ is the solution to Algorithm \ref{algo3:regular} below. 
\begin{algorithm}[H]
\caption{Regular adaptive finite element algorithm for mildly nonlinear problems}
\label{algo3:regular}
STEP 1: Find $u_{H_K^0}\in \mathcal{V}_{H_K^0}$ such that
\begin{align*}
A(u_{H_K^0},v_{H_K^0}) = 0\qquad\forall v_{H_K^0}\in \mathcal{V}_{H_K^0};
\end{align*}
STEP 2: For $k\ge0$, find $u_{H_K^{k+1}}\in \mathcal{V}_{H_K^{k+1}}$ such that
\begin{align*}
A(u_{H_K^{k+1}},v_{H_K^{k+1}}) = 0\qquad\forall v_{H_K^{k+1}}\in \mathcal{V}_{H_K^{k+1}}.
\end{align*}
\end{algorithm}


Figure \ref{fig_add2} depicts the convergence histories of the $H^1$ semi-norm error as well as the error estimator for Algorithm \ref{algo3} (top) and the regular adaptive finite element algorithm \ref{algo3:regular} (bottom). We observe the optimal convergence for both algorithms. However, Algorithm \ref{algo3} is more efficient in terms of computational cost since only a linear equation is solved except on the initial mesh. More examples can be found in the extended version of this paper in \cite{li2019analysis}.
\begin{figure}[H]
\centering
\includegraphics[height=1.4in,width=4.0in]{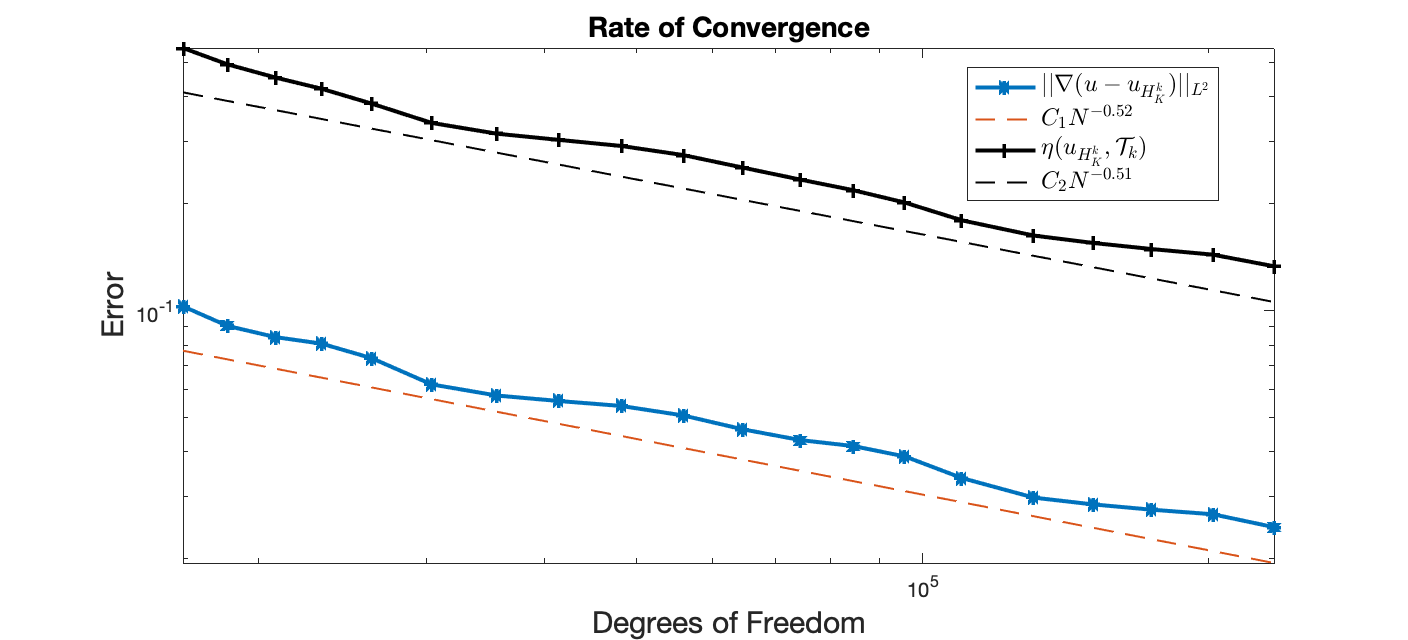}

\includegraphics[height=1.4in,width=4.0in]{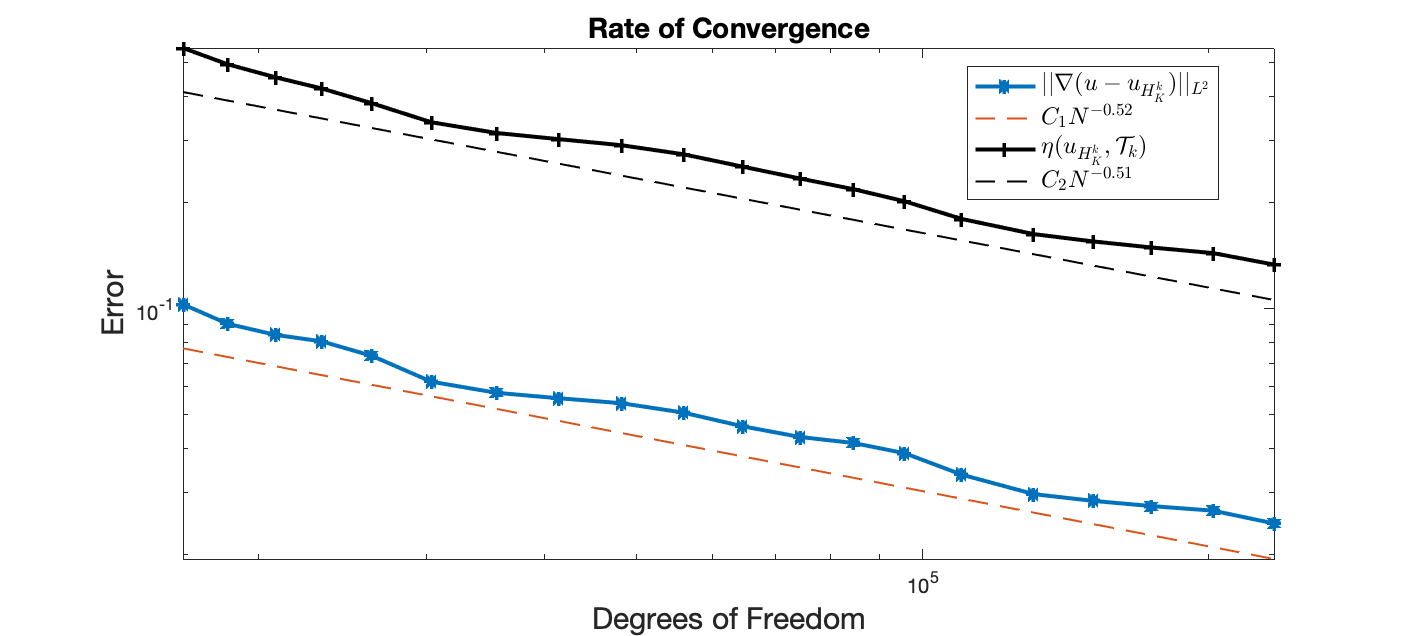}
\caption{The top is the rate of convergence for Algorithm \ref{algo3}, and the bottom is the rate of convergence for Algorithm \ref{algo3:regular}.}\label{fig_add2}
\end{figure}

For the comparison between Algorithm \ref{algo3} and two-grid algorithm \cite{xu1996two}, we assume the smallest mesh sizes are the same. Then after $20$ bisections, the degrees of freedom on the uniform mesh are more than $1.5\times10^{10}$. However, the degrees of freedom of algorithms \ref{algo3} are less than $2.4\times10^5$ as we observed, which is much more efficient.

Moreover, we numerically check higher-order terms on the right-hand side of \eqref{thm:conv4.1:est}. In Table \ref{tab1}, we observe that these terms are very small compared to the $H^1$ error, and they approach zero as bisection step $k$ or the degrees of freedom increase.
\begin{table}[!htbp]
\begin{center}
\centering
\begin{tabular}{|l|c|c|c|c|c|c|c|c|c|c|c|c|c|c|}
\hline
k & 1 & 2 & 3 & 4 & 5 & 6 & 7 & 8 & 9 & 10  \\ \hline
h.o.t. ($10^{-4}$) & $2.86$  & $2.18$& $2.13$& $1.85$& $1.54$& $1.24$& $0.99$& $0.78$& $0.62$& $0.48$\\ \hline
\end{tabular}
\caption{The higher-order terms decrease as $k$ increases.} 
\label{tab1} 
\end{center}
\end{table}

\paragraph{Test 2} In this test, we choose the domain to be $\Omega=[-1,1]\times[-1,1]$ and the exact solution to be $u=\sin\pi x\sin\pi y$, and consider mildly nonlinear PDEs below
\begin{alignat}{2}
-\mathrm{div}((2-u)\nabla u)  &= f\qquad &&\text{in}\ \Omega,\label{eq20190618_1}\\
u &= 0\qquad &&\text{on}\ \partial\Omega\label{eq20190618_2},
\end{alignat}
where $f=4\pi^2\sin\pi x\sin\pi y+\pi^2\sin^2\pi y\cos2\pi x+\pi^2\sin^2\pi x\cos2\pi y$.

Figure \ref{fig_add4} depicts the convergence histories of the $H^1$ semi-norm as well as the error estimators for Algorithm \ref{algo3} (top) and Algorithm \ref{algo5} (bottom). We observe the optimal convergence for both algorithms. More examples can be found in the extended version of this paper in \cite{li2019analysis}.
\begin{figure}[H]
\centering
\includegraphics[height=1.4in,width=4.0in]{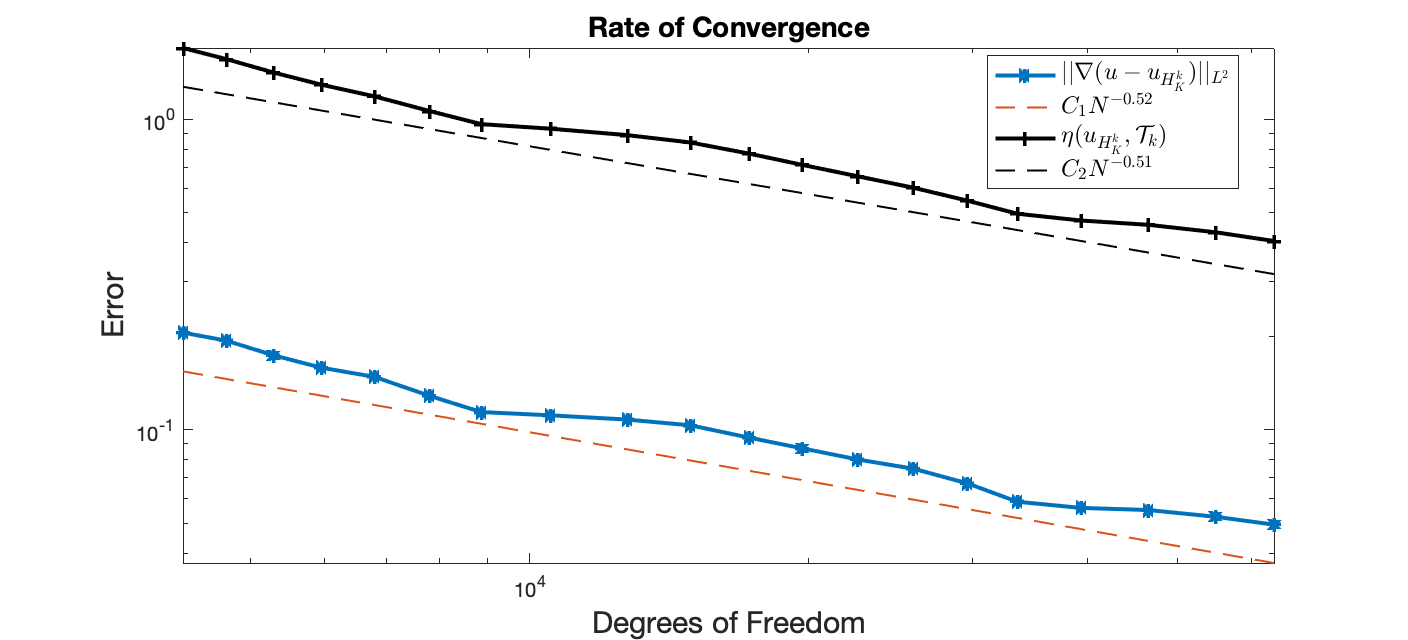}

\includegraphics[height=1.4in,width=4.0in]{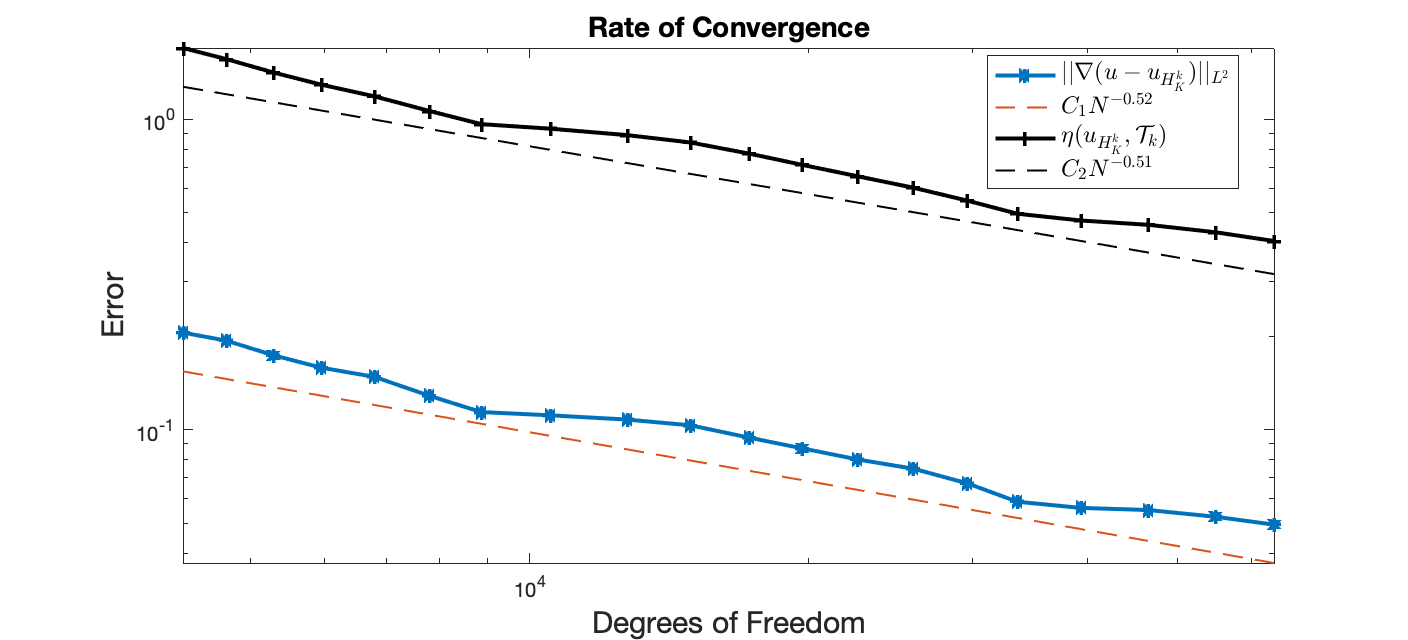}
\caption{The top is the rate of convergence for Algorithm \ref{algo3}, and the bottom is the rate of convergence for Algorithm \ref{algo5}.}\label{fig_add4}
\end{figure}

In Table \ref{tab3}, we also observe that the higher-order terms on the right-hand side of \eqref{thm:conv4.1:est} are very small compared to the $H^1$ error. They will converge to 0 as bisection step $k$ or the degrees of freedom increase.
\begin{table}[!htbp]
\begin{center}
\centering
\begin{tabular}{|l|c|c|c|c|c|c|c|c|c|c|c|c|c|c|}
\hline
k & 1 & 2 & 3 & 4 & 5 & 6 & 7 & 8 & 9 & 10  \\ \hline
h.o.t.1 ($10^{-3}$) & $1.38$  & $1.02$& $1.00$& $0.94$& $0.79$& $0.64$& $0.52$& $0.42$& $0.33$& $0.27$\\ \hline
h.o.t.2 ($10^{-3}$) & $1.59$  & $0.97$& $0.95$& $0.90$& $0.75$& $0.61$& $0.50$& $0.40$& $0.31$& $0.25$\\ \hline
\end{tabular}
\caption{The higher-order terms decrease as $k$ increases.} 
\label{tab3} 
\end{center}
\end{table}
Here h.o.t.1 and h.o.t.2 represent the higher-order terms of Algorithm \ref{algo3} and Algorithm \ref{algo5}, respectively.

\paragraph{Test 3} Consider the following nonlinear PDE with Dirichlet boundary condition:
\begin{alignat}{2}
u-k\Delta u + \frac{k}{\epsilon^2}(u^3-u) &= \tanh\bigl(\frac{d({\bf x})}{\sqrt{2}\epsilon}\bigr)-k\frac{1-\tanh^2\bigl(\frac{d({\bf x})}{\sqrt{2}\epsilon}\bigr)}{\sqrt{2(x^2+y^2)}\epsilon}\qquad &&\text{in}\ \Omega,\label{eq20180323_1}\\
u &= 1\qquad &&\text{on}\ \partial\Omega.\label{eq20180323_2}
\end{alignat}
where $d({\bf x})$ denotes the distance function from point ${\bf x}$ to the circle $x^2+y^2=0.3^2$, and domain $\Omega=[-1,1]\times[-1,1]$. The exact solution is $\tanh\bigl(\frac{d({\bf x})}{\sqrt{2}\epsilon}\bigr)$, and the boundary value is 1 due to the machine error. Here we choose $\epsilon=0.05$ and $k=\epsilon^2$. We remark that when $k$ is small, solving equation \eqref{eq20180323_1} is equivalent solving one time step Allen-Cahn equation with initial condition $\tanh\bigl(\frac{d({\bf x})}{\sqrt{2}\epsilon}\bigr)$ and time step $k$ \cite{feng2014analysis, xu2016convex}.


Figure \ref{fig_add3} depicts the convergence histories of the $H^1$ semi-norm as well as the error estimator for Algorithm \ref{algo5} (top) and the regular adaptive algorithm \ref{algo3:regular}. We observe the optimal convergence for both algorithms. Similar to Test 1, Algorithm \ref{algo3} is more efficient since only a linear equation is solved except on the initial mesh. 
\begin{figure}[H]
\centering
\includegraphics[height=1.6in,width=4.0in]{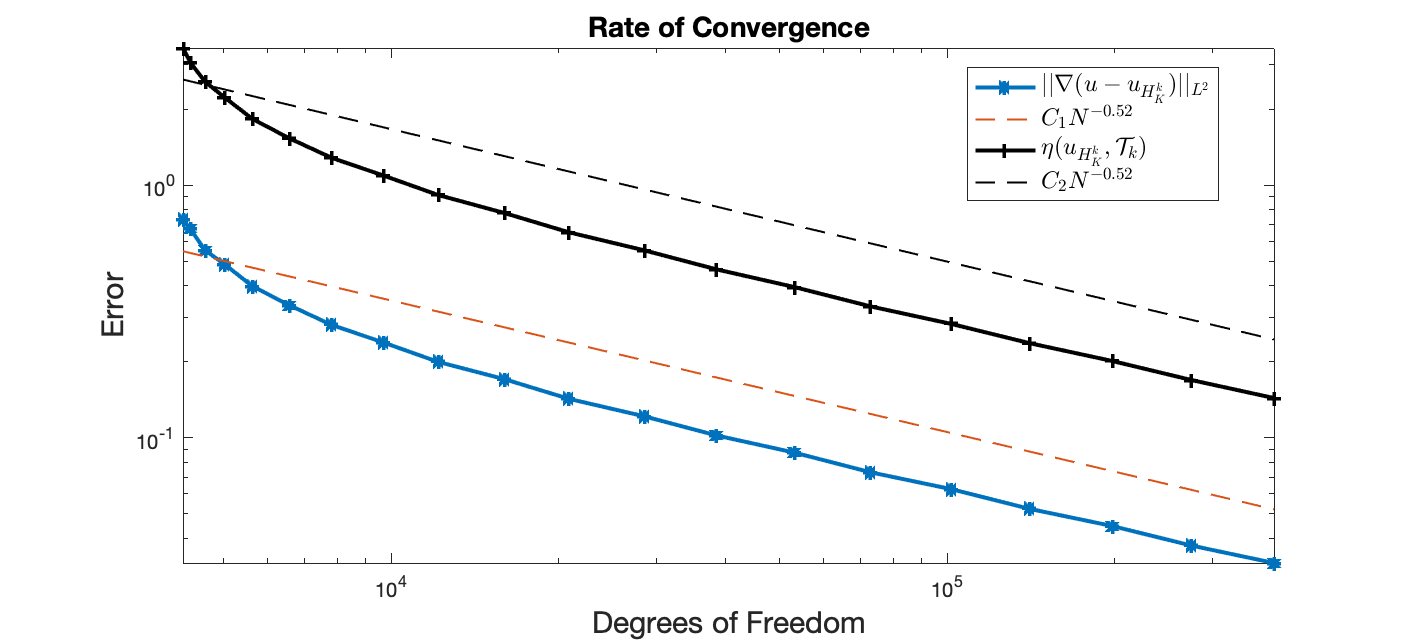}

\includegraphics[height=1.6in,width=4.0in]{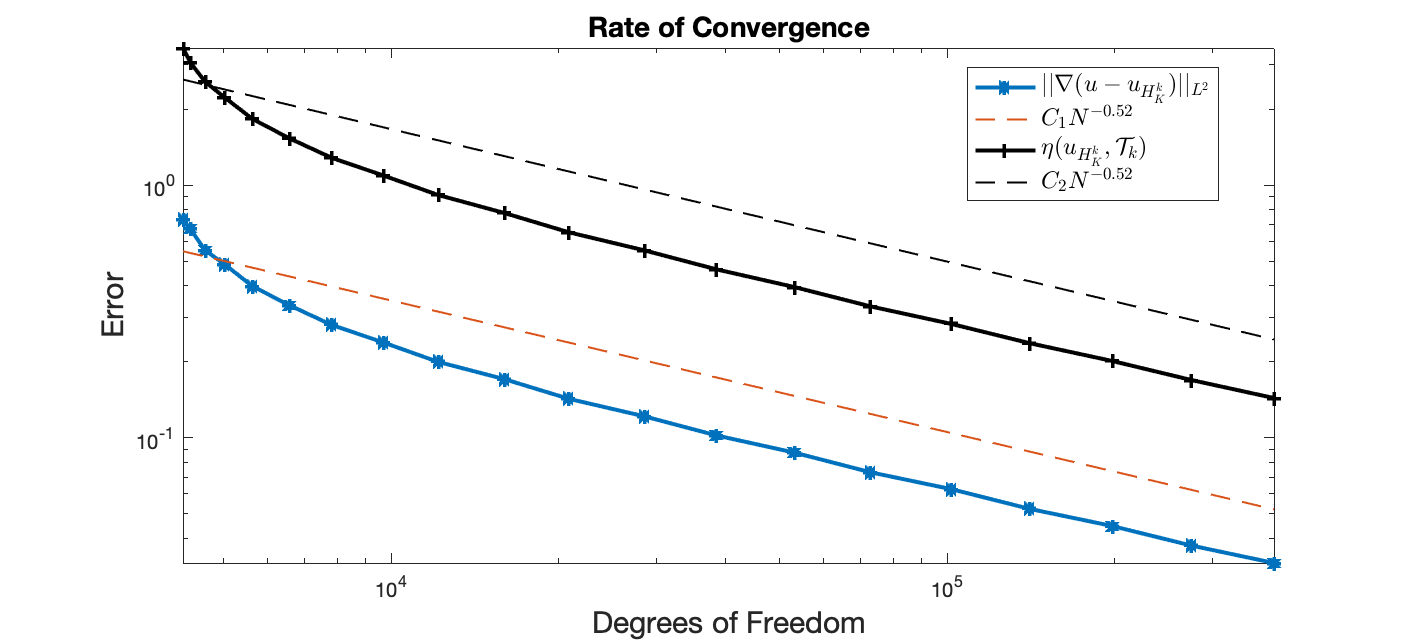}
\caption{The top is the rate of convergence for Algorithm \ref{algo5}, and the bottom is the rate of convergence for the regular adaptive algorithm \ref{algo3:regular}.}\label{fig_add3}
\end{figure}

For the comparison between Algorithm \ref{algo5} and classical two-grid algorithm \cite{xu1996two}, we assume the smallest mesh sizes to be the same. Then after $20$ bisections, the degree of freedom on the uniform meshes is more than $1.5\times10^{10}$. However, the degree of freedom of algorithms \ref{algo5} is less than $4\times10^5$ as we observed, which is much more efficient.

We also numerically check higher order terms (h.o.t.) on the right-hand side of \eqref{thm:conv4.1:est}. In Table \ref{tab2}, we verify that these terms are very small compared to $H^1$ error, and they approach 0 as bisection step k or the degree of freedom (d.o.f.) increases.
\begin{table}[!htbp]
\begin{center}
\centering
\begin{tabular}{|l|c|c|c|c|c|c|c|c|c|c|c|c|c|c|}
\hline
k & 1 & 2 & 3 & 4 & 5 & 6 & 7 & 8 & 9 & 10  \\ \hline
h.o.t. ($10^{-3}$) & $4.89$  & $3.97$& $3.88$& $3.31$& $2.68$& $2.10$& $1.63$& $1.25$& $0.95$& $0.72$\\ \hline
\end{tabular}
\smallskip
\caption{The higher order term decreases as $k$ increases.} 
\label{tab2} 
\end{center}
\end{table}

\section*{Acknowledgements}
The author highly thanks Dr. Jinchao Xu in Pennsylvania State University for his motivation and ideas in this paper, and Dr. Shuonan Wu in Peking University for his useful suggestions during the preparation of this manuscript.


\begin{thebibliography}{99}
\bibitem{bi2018posteriori}
C. BI, C. WANG AND Y.  LIN, {\em A posteriori error estimates of two-grid finite element methods for nonlinear elliptic problems,} J. Sci. Comput., 74 (2018), pp. 23--48.

\bibitem{binev2004adaptive}
P. BINEV, W. DAHMEN AND R. DEVORE, {\em Adaptive finite element methods with convergence rates,} Numer. Math., 97 (2004), pp. 219--268.

\bibitem{brenner2007mathematical}
S. BRENNER AND R. SCOTT, {\em The mathematical theory of finite element methods,} Springer Science \& Business Media, 15, 2007.

\bibitem{cascon2008quasi}
J. CASCON, C. KREUZER, R. NOCHETTO AND K. SIEBERT, {\em Quasi-optimal convergence rate for an adaptive finite element method,} SIAM J. Numer. Anal., 46 (2008), pp. 2524--2550.

\bibitem{cai2009numerical}
M. CAI, M. MU AND J. XU, {\em Numerical solution to a mixed Navier--Stokes/Darcy model by the two-grid approach,} SIAM J. Numer. Anal., 47 (2009), pp. 3325--3338.

\bibitem{congreve2014}
S. CONGREVE, {\em Two-grid hp-version discontinuous Galerkin finite element methods for quasilinear PDEs,} Ph.D. Dissertation, University of Nottingham (2014).

\bibitem{congreve2013two_2}
S. CONGREVE AND P. HOUSTON, {\em Two-grid hp-DGFEM for second order quasilinear elliptic PDEs based on an incomplete Newton iteration,} Proceedings of the 8th International Conference on Scientific Computing and Applications, 586 (2013), pp. 135--142.

\bibitem{congreve2014two}
S. CONGREVE AND P. HOUSTON, {\em Two-grid hp-version discontinuous Galerkin finite element methods for quasi-Newtonian fluid flows,} Int. J. Numer. Anal. Model., 11 (2014).

\bibitem{congreve2019two}
S. CONGREVE AND P. HOUSTON, {\em Two-Grid hp-DGFEMs on Agglomerated Coarse Meshes,}  PAMM. Proc. Appl. Math. Mech, 19 (2019).

\bibitem{congreve2013hp}
S. CONGREVE, P. HOUSTON AND T. WIHLER, {\em hp--Adaptive Two-Grid Discontinuous Galerkin Finite Element Methods for Quasi-Newtonian Fluid Flows}, Numerical Mathematics and Advanced Applications, (2013), pp. 341--349.

\bibitem{congreve2013two}
S. CONGREVE, P. HOUSTON AND T. WIHLER, {\em Two-grid hp-version discontinuous Galerkin finite element methods for second-order quasilinear elliptic PDEs,} J. Sci. Comput., 55 (2013), pp. 471--497.


\bibitem{dorfler1996convergent}
W. D{\"O}RFLER, {\em A convergent adaptive algorithm for Poisson's equation,} SIAM J. Numer. Anal., 33 (1996), pp. 1106--1124.

\bibitem{feng2014analysis}  X. FENG AND Y. Li, \textit{ Analysis of symmetric interior penalty discontinuous Galerkin methods for the Allen--Cahn equation and the mean curvature flow}, IMA J. Numer. Anal., 35 (2015), pp. 1622-1651.

\bibitem{hu2013convergence}
J. HU AND J. XU, {\em Convergence and optimality of the adaptive nonconforming linear element method for the Stokes problem,} J. Sci. Comput., 55 (2013), pp. 125--148.

\bibitem{karakashian2003posteriori}
O. KARAKASHIAN AND F. PASCAL, {\em A posteriori error estimates for a discontinuous Galerkin approximation of second-order elliptic problems,} SIAM J. Numer. Anal., 41 (2003), pp. 2374--2399.

\bibitem{karakashian2007convergence}
O. KARAKASHIAN AND F. PASCAL, {\em Convergence of adaptive discontinuous Galerkin approximations of second-order elliptic problems,} SIAM J. Numer. Anal., 45 (2007), pp. 641--665.


\bibitem{li2019analysis} Y. LI AND Y. ZHANG, \textit{Analysis of novel adaptive two-grid finite element algorithms for linear and nonlinear problems}, arXiv preprint arxiv.org/abs/1805.07887v2, (2019).

\bibitem{morin2000data}
P. MORIN, R. NOCHETTO AND K. SIEBERT, {\em Data oscillation and convergence of adaptive FEM,} SIAM J. Numer. Anal., 38 (2000), pp. 466--488.

\bibitem{morin2002convergence}
P. MORIN, R. NOCHETTO AND K. SIEBERT, {\em Convergence of adaptive finite element methods,} SIAM Rev., 44 (2002), pp. 631--658.

\bibitem{mu2007two}
M. MU AND J. XU, {\em A two-grid method of a mixed Stokes--Darcy model for coupling fluid flow with porous media flow,} SIAM J. Numer. Anal., 45 (2007), pp. 1801--1813.

\bibitem{scott1990finite}
R. SCOTT AND S. ZHANG, {\em Finite element interpolation of nonsmooth functions satisfying boundary conditions,} Math. Comp., 54 (1990), pp. 483--493.

\bibitem{verfurth1996review}
R. VERF{\"U}RTH, {\em A review of {\em a posteriori} error estimation and adaptive mesh-refinement techniques,} John Wiley \& Sons Inc, 1996.

\bibitem{verfurth2013posteriori}
R. VERF{\"U}RTH, {\em {\em A posteriori} error estimation techniques for finite element methods,} OUP Oxford, 2013.

\bibitem{xu1992iterative}
J. XU, {\em Iterative methods by space decomposition and subspace correction.} SIAM Rev., 34 (1992), pp. 581--613.

\bibitem{xu1994novel}
J. XU, {\em A novel two-grid method for semilinear elliptic equations.} SIAM J. Sci. Comput., 15 (1994), pp. 231--237.

\bibitem{marion1995error}
M. MARION AND J. XU, {\em Error estimates on a new nonlinear Galerkin method based on two-grid finite elements.} SIAM J. Numer. Anal., 32 (1995), pp. 1170--1184.

\bibitem{xu1996two}
J. XU, {\em Two-grid discretization techniques for linear and nonlinear PDEs.} SIAM J. Numer. Anal., 33 (1996), pp. 1759--1777.

\bibitem{xu2016convex}  J. XU, Y. LI, S. WU AND A. BOUSQUET, {\em On the stability and accuracy of partially and fully implicit schemes for phase field modeling.} Computer Methods in Applied Mechanics and Engineering, 345 (2019), pp. 826-853.

\bibitem{xu2001two}
J. XU AND A. ZHOU, {\em A two-grid discretization scheme for eigenvalue problems.} Math. Comp., 70 (2001), pp. 17--25.

\bibitem{zhong2013two}
L. ZHONG, S. SHU, J. WANG AND J. XU, {\em Two-grid methods for time-harmonic Maxwell equations.} Numer. Linear Algebra Appl., 20 (2013), pp. 93--111.
\end{thebibliography}
\end{document}